\documentclass[12pt]{amsart}

\usepackage[left=2.75cm,right=2.75cm,top=2.5cm,bottom=2.5cm]{geometry}
\usepackage[utf8]{inputenc} 
\usepackage[T1]{fontenc}
\usepackage{url}
\usepackage{ifthen}
\usepackage{amsmath,amssymb,latexsym,enumerate,amsthm} 
\usepackage{tikz}
\usepackage{verbatim,booktabs,blkarray} 
\usepackage{xfrac}
\usepackage[all]{xy}

\newcommand{\st}{\ast}
\newcommand{\ds}{-}

\theoremstyle{definition}
\newtheorem{theorem}{Theorem}
\newtheorem{definition}[theorem]{Definition}

\newtheorem{proposition}[theorem]{Proposition}
\newtheorem{lemma}[theorem]{Lemma}
\newtheorem{corollary}[theorem]{Corollary}
\newtheorem{example}[theorem]{Example}
\newtheorem{remark}[theorem]{Remark}

\DeclareMathOperator{\sign}{sign}

\begin{document}

\title{Johnson Graph Codes}

\author{ {Iwan~Duursma and Xiao~Li} }
\thanks{ This work was supported in part by NSF grant CCF 1619189. Email: {duursma,~xiaoli17~@illinois.edu.} }

\maketitle

\begin{abstract}
We define a Johnson graph code
as a subspace of labelings of the vertices in a Johnson graph with the property that labelings are uniquely determined by their restriction to vertex neighborhoods specified by the parameters of the code. We give a construction and main properties for the codes and show their role in the concatenation of layered codes that are used in distributed storage systems.
A similar class of codes for the Hamming graph is discussed in an appendix. Codes of the latter type are in general different from affine Reed-Muller codes, but for the special case of the hypercube they agree with binary Reed-Muller codes. 
\end{abstract}

\section{Introduction}
Codes for distributed storage encode data as a collection of $n$ smaller pieces that are then stored at $n$ different nodes, possibly in different locations. 
A code is regenerating of type $(n,k,d)$ if the original data can be recovered from any $k$ of the $n$ pieces, i.e., by downloading data from any $k$ of the $n$ nodes, and if any failed node can be repaired with small amounts of helper data from any $d$ of the remaining nodes~\cite{Dimakis+10}. The need for special codes for distributed storage arises because classical error-correcting codes that meet the requirement for efficient data collection, including Reed-Solomon codes, do not have efficient node repair.  A code is called exact repair code if a failed node is replaced with an identical copy. It is called functional repair code if a failed node is replaced with a node that maintains the data collection and node repair properties of the code.

Layered codes \cite{Tian+14} are an important building block for several families of regenerating codes, including improved layered codes \cite{Senthoor+15} and coupled-layer codes \cite{Sasidharan+}. These codes are defined with $d=n-1$. Determinant codes \cite{Elyasi+16} and cascade codes \cite{Elyasi+18},  \cite{Elyasi+19} use a different construction that removes this constraint. For a distributed storage system with $n$ nodes and for a parameter $2 \leq v \leq n$, a layered code is defined by ${n \choose v}$ disjoint checks of length~$v$. The symbols $(c_i : i \in L)$ in a layer $L \subset \{0,1,\ldots,n-1\}$ of size $v$ satisfy $\sum_{i \in L} c_i = 0$. With a ratio of $v-1$ data symbols to $1$ redundant symbol, the use of longer checks makes a code more storage efficient. 
At the same time node repair becomes more expensive, with a ratio of $v-1$ helper symbols for each erased symbol.
Figure \ref{fig433} illustrates different trade-offs for $n=4$ nodes. In general, for codes of type $(n,k,d)=(n,n-1,n-1)$, layered codes and their space sharing versions realize the optimal trade-offs between storage overhead and repair bandwidth for exact repair regenerating codes \cite{T14}, \cite{MT15}, \cite{PK15}, \cite{D18}. 
Of special importance are the MSR point (minimum storage) for $v=n$, and the MBR point (minimum bandwidth) for $v=2$.

In this paper we define Johnson graph codes and show how they can be used to construct new regenerating codes by combining different layered codes of different sizes. The Johnson graph codes add relations among data in different layers. With the added relations, data can be reconstructed after connecting
to a smaller amount of nodes. The combination of layered codes does not affect node repair properties. The new regenerating codes have parameters
$(n,k,d)$ with $k < d = n-1$, whereas for layered codes $k = d = n-1$. The improved data download from fewer than $n-1$ nodes comes at a cost of an increased amount of redundant data at each node. The construction includes as special cases improved layered codes \cite{Senthoor+15} and cascade codes with $d=n-1$. In \cite{Elyasi+18}, cascade codes are defined for general $(n,k,d)$ in terms of determinant codes \cite{Elyasi+16}. It is conjectured in
\cite{Elyasi+18} that cascade codes realize optimal storage versus bandwidth trade-offs for all $(n,k,d)$. It is not hard to see that cascade codes are optimal 
for all $(n,k,d)$ if and only if they are optimal for the special case $d=n-1$.

The outline of the paper is as follows. Layered codes are presented in Section \ref{S:lc}. Johnson graph codes are defined in Section \ref{S:Jdefn} as subspaces of labelings of the vertices in a Johnson graph with the property that labelings are uniquely determined by their restriction to vertex neighborhoods specified by the parameters of the code.
A general construction for the codes is presented in Section \ref{S:Jcons}. Section \ref{S:dual} describes equivalent constructions and addresses duality.
As a special case of the general construction, in Section \ref{S:rs} we construct Johnson graph codes from Reed-Solomon codes. Section \ref{S:appl} presents a construction of regenerating codes as a concatenation of layered codes. The construction of Johnson graph codes may be pursued for other graphs, and for distance regular graphs in particular. The case of Hamming graphs is of particular interest and is discussed in an appendix.   

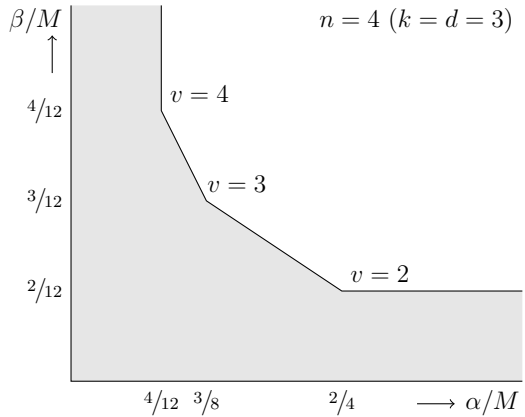
\begin{figure} \label{fig433}
\begin{center}
\begin{tikzpicture}  
  [scale=1] 
 
 \node [left,scale=0.8] (y1) at (0,1.2) {$\sfrac2{12}$};
 \node [left,scale=0.8] (y2) at (0,3.6) {$\sfrac4{12}$};

 \node [left,scale=0.8] at (0,4.8) {$\beta/M$};
 \path [draw=black,->] (-.25,4.1) -- (-.25,4.6); 

 \node [below,scale=0.8] (x1) at (1.2,0) {$\sfrac4{12}$};
 \node [below,scale=0.8] (x2) at (3.6,0) {$\sfrac24$};

 \node [below,scale=0.8] at (5.6,0) {$\alpha/M$};
 \path [draw=black,->] (4.6,-0.3) -- (5.1,-0.3); 
 
 \node [below,scale=0.8] (x3) at (1.8,0) {$\sfrac38$};
 \node [left,scale=0.8] (y3) at (0,2.4) {$\sfrac3{12}$};
 
 \node [above right,scale=0.8] at (1.2,3.6) {$v=4$}; 
 \node [above right,scale=0.8] at (3.6,1.2) {$v=2$}; 
 
 \node [above right,scale=0.8] at (1.7,2.4) {$v=3$}; 

\path [fill=black!10] (5,0) -- (0,0) -- (0,5) -- (1.2,5) -- (1.2,3.6) -- (1.8,2.4) -- (3.6,1.2) -- (6,1.2) -- (6,0);
\path [draw=black] (6,0) -- (0,0) -- (0,5);
\path [draw=black] (1.2,5) -- (1.2,3.6) -- (1.8,2.4) -- (3.6,1.2) -- (6,1.2);

\node [left,scale=0.8] at (6,4.8) {$n=4~(k=d=3)$};

\end{tikzpicture}
\end{center}
\caption{Storage cost $\alpha/M$ vs Node repair cost $\beta/M$ for $n=4$ nodes using local checks of length $v=2,3,4$.}
\end{figure}

\section{Layered codes} \label{S:lc}

Layered codes are defined in \cite{Tian+14}. 
The layered codes that we consider in this section are called canonical layered codes in that paper. 
For given $n$ and $2 \leq v \leq n$ a layered code is constructed as follows. For each $v$-subset $L \subset \{ 0, 1, \ldots, n-1 \}$ (called a layer) a node $i \in L$ stores a symbol $c_i$ such that $\sum_{i \in L} c_i = 0$. The ${n \choose v}$ local codewords, one for each layer, are concatenated into one codeword of length ${n \choose v} v.$ With one parity per layer, the overall redundancy of the code is $R = {n \choose v}$. Let $\alpha$ be the number of symbols stored by each of the $n$ nodes. Counting the total number of symbols in the code in two ways, by node and by layer, gives $n \alpha = R v.$ Let $\beta$ be the number of symbols that is downloaded from a helper node to repair a failed node. There is one such symbol for each layer that contains both the helper node and the failed node. Counting the total amount of helper information in two ways, by pairs of nodes and by layer, gives ${n \choose 2} \beta = R {v \choose 2}.$

\begin{example}[Motivating example]
Consider the layered code with $(n,k,d) = (8,7,7)$ and $v = 5$. It has $(R,\alpha,\beta) = (\binom85,\binom74,\binom63) = (56,35,20).$ The code
has $\binom{8}{5} = 56$ layers, each containing five symbols that sum to zero. The symbols may be put in a $8 \times 56$ array format such that rows correspond to disks and columns to layers. Symbols $c_1, c_3, c_5, c_6, c_7$ for the layer
$\{ 1, 3, 5, 6, 7 \}$ will form a column $(0, c_1, 0, c_3, 0, c_5, c_6, c_7)^T$. Apart from the assigned zeros in the array, the only checks for the array are that each column sums to zero. When a layer is accessed in 7-out-of-8 disks either all symbols 
in a layer are available or one symbol is missing but can be recovered using the zero check sum for the layer. In general, upon choosing any seven disks, say the first seven, the layers are of two types (Table \ref{tbl:78}).
The $\binom75=21$ layers in the first group are fully accessed and all their symbols are available. The $\binom74=35$ layers in the second group miss one symbol that is however protected by the parities on the individual layers. Once we lower $k$ , the partitioning of the layers changes (Table \ref{tbl:48}). 

\begin{table} 
\begin{center}
$\begin{array}{
                       c@{\hspace{3mm}}c
                       c@{\hspace{2.5mm}}c@{\hspace{2.5mm}}c@{\hspace{2.5mm}}c@{\hspace{2.5mm}}c@{\hspace{2.5mm}}c@{\hspace{2.5mm}}c@{\hspace{3mm}}c
                       c@{\hspace{3mm}}cc}
\toprule
& &\multicolumn{7}{c}{\text{accessed nodes}} & &\multicolumn{1}{c}{\text{missed}} & &\text{distance}   \\
\midrule 
{\#}  &   &\multicolumn{7}{c}{0, 1, 2, 3, 4, 5, 6} & &\multicolumn{1}{c}{7} & &r  \\ 
\midrule
21 &     &\ds &\ds &\st  &\st&\st &\st &\st & &\ds & & 0  \\
35 &     &\ds &\ds &\ds  &\st&\st &\st &\st & &\st & & 1  \\
\bottomrule 
\end{array}$ \\
\bigskip
 \caption{Accessing a code with layers of size $v=5$ in $k=7$ nodes.}  \label{tbl:78}
\end{center}
\end{table}

\begin{table}  
\begin{center}
$\begin{array}{
                       c@{\hspace{3mm}}c
                       c@{\hspace{2.5mm}}c@{\hspace{2.5mm}}c@{\hspace{1mm}}cc@{\hspace{3mm}}c
                       c@{\hspace{2.5mm}}c@{\hspace{2.5mm}}c@{\hspace{2.5mm}}c@{\hspace{3mm}}c
                       c@{\hspace{3mm}}cc}
\toprule
& &\multicolumn{4}{c}{\text{accessed nodes}} & &\multicolumn{4}{c}{\text{missed}} & &\text{distance}   \\
\midrule 
{\#}  &   &\multicolumn{4}{c}{0, 1, 2, 3} & &\multicolumn{4}{c}{4, 5, 6, 7} & &r  \\ 
\midrule
4  & &\quad \st &\st &\st &\st & &~\st &\ds &\ds &\ds & & 0  \\
24 &  &\quad \ds &\st &\st &\st & &~\st &\st &\ds &\ds & & 1  \\
24 &  &\quad \ds &\ds &\st &\st & &~\st &\st &\st &\ds & & 2 \\
4  & &\quad \ds &\ds &\ds &\st & &~\st &\st &\st &\st & & 3  \\
 \bottomrule
 \end{array}$ 
 \bigskip
 \caption{Accessing a code with layers of size $v=5$ in $k=4$ nodes.} \label{tbl:48}
\end{center}
\end{table}

The partitioning of the layers can be captured by a partitioning of vertices in a Johnson graph. The distances in Table \ref{tbl:78} and Table \ref{tbl:48} are the distances induced by the Johnson graph. Data recovery proceeds sequentially and starts with layers at $r=0$. As $r$ increases, missing data in layers of type $r$ will be recovered from available data in combination with stored parities. In the next section we define a family of codes on Johnson graphs that, for a minimum of stored extra parities and for any combination of $k$ accessed disks, recover data in under-accessed layers from available data in fully-accessed layers. The recovery process itself is described in Section \ref{S:appl}. 
\end{example}

\section{Johnson Graph Codes} \label{S:Jdefn}

\begin{definition}
The Johnson graph $J(n,v)$ is the undirected graph with vertex set all $v$-subsets of a given $n$-element set such that two vertices are adjacent if they intersect in $v-1$ elements. 
The default choice for an $n$-element set is the set $\{0,1,\dots, n-1\}$. 
\end{definition}

Our goal is to construct a linear code of labelings of the vertices in a Johnson graph such that labelings are uniquely determined by their restriction to a neighborhood in the graph. In this section we formulate the requirements for the code. In the following three sections we give a general construction and we establish important properties. In Section \ref{S:appl} the codes are used in the construction of exact repair regenerating codes for use in distributed storage. 

With the usual definition of distance between two vertices as the length of the shortest path between two vertices, two vertices in $J(n,v)$ are at distance $r$ if they intersect in 
$v-r$ elements. 
For any given vertex $x$, the vertices in $J(n,v)$ partition into subsets according to their distance to $x$. We refer to the subsets as the shells around $x$. 

\begin{example} \label{ex:5221}
The shells around $x = \{ 0,1 \} \subset \{ 0,1,2,3,4 \}$ are (we write $x = 01$)
\begin{align*}
S_0(01) = \{ 01 \},  ~~
S_1(01) = \{ 02, 03, 04, 12, 13, 14 \}, ~~
S_2(01) = \{ 23, 24, 34 \}.
\end{align*}
We construct a code of length 10 for the graph $J(5,2)$ that recovers the three symbols at the vertices $S_2(x)$ from the seven symbols at the vertices $S_0(x)$ and $S_1(x)$, for any choice of vertex $x$, by storing three fixed parities. Let $\alpha_0, \alpha_1, \alpha_2, \alpha_3, \alpha_4$ be five distinct elements of the field~$F$. 
For a general word $c = (c_{ij}) \in F^{10}$ define a parity vector $H c \in F^3$ as 
\[
H c = \sum c_{ij} \left( \begin{array}{c} 1 \\ \alpha_i + \alpha_ j \\ \alpha_i \alpha_j \end{array} \right).
\]
For each of the ten sets $S_2(x)$, where $x$ is a vertex in $J(5,2)$, the corresponding columns in the parity matrix $H$ are linearly independent. 
For $S_2(x) = \{ 23, 24, 34 \}$, the three columns form the invertible submatrix 
\[
\left( \begin{array}{ccc} 1 &1 &1 \\ \alpha_2 + \alpha_ 3 &\alpha_2 + \alpha_ 4 & \alpha_3 + \alpha_ 4 \\ 
\alpha_2 \alpha_3 &\alpha_2 \alpha_4 &\alpha_3 \alpha_4 \end{array} \right).
\]
Thus the three symbols $c_{23}, c_{24}, c_{34}$ can be recovered from the remaining seven symbols in combination with the three fixed parities $H c$. Moreover, the three symbols can be recovered efficiently one at a time. The equation $(\alpha_4^2, - \alpha_4, 1)Hc$ expresses $c_{23}$ as a linear combination of $c_{01}, c_{02}, c_{03}, c_{12}, c_{13}.$ The symbols $c_{24}, c_{34}$ are recovered similarly.

\end{example}

We encounter Johnson graphs in two different ways. For a storage system with $n$ nodes and for a parameter $2 \leq v \leq n$, layered codes are defined by ${n \choose v}$ disjoint checks of length~$v$, one for each $v$-subset of nodes. The $v$-subsets are called layers.
They form a set of vertices for a Johnson graph $J(n,v)$. A second graph $J(n,k)$ occurs as follows. For a regenerating code of type $(n,k,d)$ data is collected by contacting any $k$ nodes. The possible data collection scenarios are the $k-$subsets of an $n-$element set. They form the vertices of a Johnson graph $J(n,k)$.

For each collection scenario, i.e., for each $k$-subset $A \subset \{0,1,\ldots,n-1\}$, we define a partition of the layers in the layered code, i.e., of the vertices in $J(n,v)$. Table~\ref{tbl:78} and Table~\ref{tbl:48} illustrate vertex partitions into shells for the cases $k > v$ and $k < v$. 

For $k \geq v$, layers $L$ with $L \subset A$ 
are fully accessible. They form the shell $S_0(A) = \{ L : L \subset A \}.$ Layers at distance $r$ from $S_0(A)$ 
 form the shell $S_r(A) = \{ L : |L \backslash A| = r \}.$ 
For $k \leq v$, layers are accessible in at most $k$ symbols. The maximally accesible layers are those with $A \subset L$. They form the shell $S_0(A) = \{ L : L \supset A \}.$
The shell at distance $r$ from these layers is $S_r(A) = \{ L : |A \backslash L| = r \}.$ In general, since $k-v = |A \backslash L| - |L \backslash A|$,
\begin{equation} \label{eq:Sr}
L \in S_r(A) ~~\Leftrightarrow~~ r = \min( \lvert L \backslash A \rvert, \lvert A \backslash L \rvert ).
\end{equation}
For $L \in J(n,v)$, let $B_r(L)$ be the neighborhood of radius $r$ around $L$.

\begin{definition} \label{def:Br}
For the Johnson graph $J(n,v)$ and for a subset $A$ of $\{0,1,\dots, n-1\}$, 
define a generalized neighborhood as a union of vertex neighborhoods.
\[
B_r(A) = \begin{cases} \cup_{L_0 \subset A} B_r({L_0}) ~&(|A| \geq v). \\
\cup_{L_0 \supset A} B_r(L_0) ~&(|A| \leq v). 
\end{cases}
\]
\end{definition}

It is easy to see that membership $L \in B_r(A)$ depends only on the size of $L \backslash A$ or $A \backslash L$.
\begin{equation} \label{eq:Br}
L \in B_r(A) ~~\Leftrightarrow~~ \begin{cases} \lvert L \backslash A \rvert \leq r ~&(|A|\geq v)  \\ 
 \lvert A \backslash L \rvert \leq r ~&(|A| \leq v)  
\end{cases}
~~\Leftrightarrow~~ r \geq \min( \lvert L \backslash A \rvert, \lvert A \backslash L \rvert ). 
\end{equation}
The shell $S_r(A) = B_r(A) \backslash B_{r-1}(A)$. 
With the above notions in mind, we define Johnson graph codes as codes on the vertices of a Johnson graph such that generalized neighborhoods form information sets.

\begin{definition} \label{d1}
A Johnson graph code $JGC(n, v, k, r)$ on the vertices of the Johnson graph $J(n, v)$ is a code of length $N = \binom{n}{v}$ and dimension 
$K = \lvert B_r( A ) \rvert$ such that, for any $k$-subset $A$ of $\{0,1,\dots, n-1\}$, $B_r(A)$ forms an information set. 
\end{definition}

It is worth noting that the definition only refers to matroid properties of the code. In matroid terms it asks that the set of bases for the matroid of the code 
is large enough that it includes all generalized neighborhoods $B_r(A)$. The definition includes as trivial examples all MDS codes of length $N$ and dimension $K$. Any Johnson graph code, and in particular any length $N$, dimension $K$ MDS code, can be used in the construction of concatenated layered codes in Section \ref{S:appl}. 
The codes constructed in the next section have the advantage of being defined over a field of
size $n$ instead of $N = {n \choose v}$ and moreover have a parity check structure that allows efficient erasure decoding.  
We include some lemmas for later use and prove in Proposition \ref{p:dualjgc} that the dual of a Johnson graph code is again a Johnson graph code.

\begin{lemma} Let $k = |A|$.
\[
\lvert B_r(A)\rvert  = \begin{cases}
\sum_{i=0}^{r} \binom{n-k}{i}\binom{k}{v-i} ~&(k \geq v). \\
\sum_{i=0}^{r}\binom{k}{i}\binom{n-k}{n-v-i} ~&(k \leq v). 
\end{cases}
\]
\end{lemma}

\begin{proof}
The first sum counts $L$ with $|L \cap A^c| = i \leq r$ and $|L \cap A| = v-i$. The second sum counts $L$ with $|L^c \cap A| = i \leq r$
and $|L^c \cap A^c| = n-v-i.$ 
\end{proof}

We introduce the following notation. Let
\begin{equation} \label{Rd1d2}
d_1 = \min( v, n-k ), ~d_2 = \min( k, n-v),
\end{equation}
and let $R = \min (d_1,d_2).$

\begin{lemma} \label{L:D}
$S_r(A) = \emptyset$ for all $r > R.$
\end{lemma}

\begin{proof} Use (\ref{eq:Sr}) with $\min( |L \backslash A|, |A \backslash L| ) \leq \min (d_1, d_2) = R.$ 
\end{proof}

\begin{lemma} \label{L0}
$S_r(A) = S_{R-r} (A^c).$
\end{lemma}

\begin{proof} For $L \in J(n,v)$, let
\[
\left( \begin{array}{cc} a &b \\ c &d \end{array} \right) =
\left( \begin{array}{cc}
\lvert L \cap A^c \rvert &\lvert L \cap A \rvert \\
\lvert L^c \cap A^c \rvert &\lvert L^c \cap A \rvert
\end{array} \right)
\]
Then $L \in S_r(A)$ for $r = \min ( a, d )$ and $L \in S_s(A^c)$ for
$s = \min ( b, c ).$ 
Where $r$ and $s$ are related to $R$ via
$r+s = \min ( a+b, a+c, b+d, c+d ) 
      = \min ( v, n-k, k, n-v ) = R.$
\end{proof}

The lemma gives a dual description for the shell $S_r(A)$ in terms of the complement $A^c$ of~$A$. 

\begin{example} The four layers $S_0(0123) = \{ 01234, 01235, 01236, 01237 \} \subset J(8,5)$ in the top row of Table \ref{tbl:48} can be described dually as $S_3(4567).$ 
\end{example}

A different dual description is obtained by considering the set $\{ L^c : L \in S_r(A) \}$ as a subset of vertices in the Johnson graph $J(n,n-v)$. Note that this graph is isomorphic to $J(n,v)$ with only difference that the vertex labels $L \subset \{ 0, 1, \ldots, n-1 \}$ are replaced with their complements $L^c =  \{ 0, 1, \ldots, n-1 \} \backslash L$. Two $v$-subsets are adjacent in $J(n,v)$ if and only if their complements are adjacent in $J(n,n-v)$. 
 
\begin{lemma} \label{Ld}
For shells $S_r(A)   \subset J(n,v)$ and $S_{R-r}(A) \subset J(n,n-v)$,
\[
 L \in S_r(A) ~\Leftrightarrow~ L^c \in S_{R-r}(A).
\]
\end{lemma}
 
 \begin{proof}
As in the proof of Lemma \ref{L0}, $L \in S_r(A)$ for $r = \min(a,d)$ and $L^c \in S_s(A)$ for $s = \min( c,b ).$ So that $r+s = R$.
 \end{proof}
 
\begin{lemma} \label{L1}
\[
B_r(A)^c = B_{R-r-1}(A^c).
\]
\end{lemma}

\begin{proof}
Use that $B_r(A)^c = \cup_{t > r} S_t(A)$ and apply Lemma \ref{L0}.
\end{proof}

\begin{proposition} \label{p:dualjgc}
The dual code of a Johnson graph code $JGC(n,v, k, r)$ is a code $JGC(n,v,n-k,R-r-1)$.
\end{proposition}

\begin{proof} The code has the required length and dimension. Lemma \ref{L1} and the fact that the complement of an information set is an information set for the dual code, proves that $B_{R-r-1}(A^c)$ is an information set for any $k$-subset $A$.
\end{proof}

\section{Construction of Johnson Graph Codes} \label{S:Jcons}

We give a general construction for Johnson graph codes of length $N = {n \choose v}$ from MDS codes of length $n$. Section \ref{S:rs} describes in more detail the special case where the length $n$ MDS code is of Reed-Solomon type. Codewords of length $N$ will have their coordinates indexed by the vertices of the Johnson graph $J(n,v)$. The codes are generated by the coordinatization vectors of suitably chosen matrices $M$.
For a $v \times n$ matrix $M$, 
\[
\pi ( M ) = ( \det( M_L ) : L \in J(n,v) ),
\]
 where $M_L$ is the $v \times v$ minor of $M$ with columns in the $v$-subset $L \subset \{ 0, 1, \ldots, n-1 \}$.

\begin{definition} \label{def:code} For given $n$ and $v$, let $C_0$ be a code of length $n$ and let $t \geq 0$. Define a code $C$ of length $N$ as the span of the vectors $\pi(M)$ for all $v \times n$ matrices $M$ with at least $t$ rows in $C_0$.
\end{definition}

We show in two steps that the code $C$ is a Johnson graph code $JGC(n,v,k,r)$ with $k = \dim C_0$ and $r = \min ( v, \dim C_0) -t.$

\begin{lemma} \label{L:jgc1}
The code $C$ has dimension $K =  \lvert B_r(I_0) \rvert,$
where $r = \min ( v, \dim C_0) -t$ and $I_0 \subset \{ 0, 1, \ldots, n-1 \}$ is of size $\dim C_0$.
\end{lemma}

\begin{proof} Let $g = \{ g_0, g_1, \ldots, g_{n-1} \}$ be a basis for $F^n$ and let
$\{ g_i : i \in I_0 \}$ be a basis for $C_0$. The vectors $\pi(M)$, for matrices $M$ with rows $\{ g_i : i \in L \}$
such that $\lvert L \cap I_0 \rvert \geq t$, form a basis for $C$. Furthermore,
\[
\lvert L \cap I_0 \rvert = t ~~\Leftrightarrow~~ \min ( \lvert L \backslash I_0 \rvert, \lvert I_0 \backslash L \rvert ) = \min ( v-t, \dim C_0 -t ).
\]
And thus, using (\ref{eq:Br}) with $r = \min ( v, \dim C_0) -t,$ 
\[
\{ \; L \in J(n,v) ~|~ \lvert L \cap I_0 \rvert \geq t \; \} = B_r(I_0).
\]
\end{proof} 

From the lemma it is clear that the dimension of the code $C$ only depends on $\dim C_0$ and not on $C_0$. The occurrence of information sets in $C$ on the other hand depends on the choice of the code $C_0$.  

\begin{lemma} \label{L:jgc2}
For $C$ as in the definition, and for an information set $A_0 \subset  \{ 0, 1, \ldots, n-1 \}$ for $C_0$, the set
\[
 \{ \; L \in J(n,v) ~|~ \lvert L \cap A_0 \rvert \geq t \; \} = B_r(A_0)
\]
is an information set for $C$.
\end{lemma}

\begin{proof} The size of $B_r(A_0)$ equals the dimension of $C$. It therefore suffices to prove the independence of the coordinates in~$B_r(A_0)$.
For each $L \in B_r(A_0)$ 
construct a codeword $\pi(M)$ from a $v \times n$ matrix $M$ with two types of rows. For $j \in L \cap A_0$ include as row in $M$ the unique codeword in $C_0$ that is $1$ in $j$ and $0$ in $A_0 \backslash j.$ For $j \in L \backslash A_0$ include as row in $M$ the unit vector that is $1$ in $j$ and $0$ elsewhere.
Upto a permutation of columns that puts the columns in $A_0 \backslash L$ in the leading positions, followed by columns in $L \cap A_0$ and $L \backslash A_0$, $M$ will be of the form
\[
M = \left( \begin{array}{c|c|c|c}
O &I &X &Y \\ \midrule
O &O &I &O 
\end{array}
\right), \quad M_L = \left( \begin{array}{c|c}
I &X \\ \midrule
O &I 
\end{array}
\right).
\]
Clearly, $\det (M_L) =1$.  And $\det (M_{L'}) = 0$ for $L' \neq L$ with $\lvert L' \cap A_0 \rvert \geq \lvert L \cap A_0 \rvert$. The constructed codewords $\pi(M)$, for $L \in S_r(A_0), \ldots, S_1(A_0), S_0(A_0)$, in that order, form an invertible triangular matrix in the positions $B_r(A_0)$. And thus $B_r(A_0)$ is an information set for $C$.  
\end{proof}

\begin{proposition} \label{prop:jgc}
For an MDS code $C_0$, the code $C$ in Definition \ref{def:code} is a Johnson graph code.
\end{proposition}

\begin{proof} 
Lemma \ref{L:jgc1} and Lemma \ref{L:jgc2}.
\end{proof}

\begin{lemma} \label{L:jgc3} For each $L \in S_r(A_0),$ the proof of Lemma \ref{L:jgc2} constructs a codeword $\pi(M)$ that depends on $L$. It has as special property that the coordinate $\pi(M)_L$ is the unique nonzero coordinate in the positions $B_r(A_0)$. Moreover the weight of $\pi(M)$ is at most 
 \[
{ 2t+n-\dim C_0-v \choose t }. 
\]
\end{lemma}

\begin{proof} For $r = \min ( v, \dim C_0) -t,$ we have, as in the proof of Lemma \ref{L:jgc1}, that $L' \in B_r(A_0)$ if and only if $\mid L' \cap A_0 \mid \geq t$. For $L \in S_r(A_0)$, equality holds and $\mid L \cap A_0 \mid = t.$ It follows that $\lvert L' \cap A_0 \rvert \geq \lvert L \cap A_0 \rvert$ for all $L' \in B_r(A_0)$, and thus, as in the proof of Lemma \ref{L:jgc2}, that $\det (M_{L'}) = 0$ for all $L' \in B_r(A_0)$, $L' \neq L.$
For the weight of $\pi(M)$, note that a full minor $M_{L'}$ in the matrix $M$ is nonsingular only if $M_{L'}$ contains all columns in $L \backslash A_0$ and no columns in $A_0 \backslash L$, i.e., only if
\[
L \cap A_0^c \subset L' \subset L \cup A_0^c. 
\]
For $\mid L \cap A_0 \mid = t$, $\mid L \cap A_0^c \mid = v-t$ and $\mid L \cup A_0^c \mid = n-\dim C_0+t.$
The number of $L'$ such that $M_{L'}$ is nonsingular is therefore at most
\[
{ 2t+n-\dim C_0-v \choose t }. 
\]
\end{proof}

\begin{example}
For the Johnson graph $J(n=5,v=3)$, let 
\[
g = \left( \begin{array}{c}
g_0 \\ g_1 \\ g_2 \\ g_3 \\ g_4 \end{array} \right) = \left( \begin{array}{ccccc}
1 &0 &1 &1 &0 \\ 0 &1 &0 &1 &1 \\ 0 &0 &1 &0 &0 \\  0 &0 &0 &1 &0 \\ 0 &0 &0 &0 &1 \end{array} \right).
\]
Let $C_0 = \text{span}(g_0, g_1)$, $t=2$. Then $r = 0$, $I_0 = \{ 0, 1 \},$ $B_0(I_0)  =
\{ 012, 013, 014 \}$, and the code $C$ is spanned by the first block of rows in the matrix 
\[
 \begin{array}{cc ccl | rccccl | r}
012 &~~ &1&\cdot&\cdot~  &~1&1&\cdot&1&0&\cdot~ &~1 \\
013 &~~ &\cdot&1&\cdot~  &~0&\cdot&1&1&\cdot&0~ &~1  \\
014 &~~  &\cdot&\cdot&1~  &~\cdot&0&1&\cdot&1&1~ &~1  \\  \cmidrule{3-12}
023 &~~ &\cdot&\cdot&\cdot~  &~1&\cdot&\cdot&\cdot&\cdot&\cdot~ &~0  \\
024 &~~  &\cdot&\cdot&\cdot~  &~\cdot&1&\cdot&\cdot&\cdot&\cdot~ &~1  \\
034 &~~  &\cdot&\cdot&\cdot~  &~\cdot&\cdot&1&\cdot&\cdot&\cdot~ &~1  \\
123 &~~  &\cdot&\cdot&\cdot~  &~\cdot&\cdot&\cdot&1&\cdot&\cdot~ &~1 \\
124 &~~  &\cdot&\cdot&\cdot~  &~\cdot&\cdot&\cdot&\cdot&1&\cdot~ &~1  \\ 
134 &~~  &\cdot&\cdot&\cdot~  &~\cdot&\cdot&\cdot&\cdot&\cdot&1~ &~0  \\  \cmidrule{3-12}
234 &~~  &\cdot&\cdot&\cdot~  &~\cdot&\cdot&\cdot&\cdot&\cdot&\cdot~ &~1  \\
\end{array}
\]

Matrix entries $\cdot$ are $0$ independent of the choice of the code $C_0$. Lemma \ref{L:jgc3} gives the number of
remaining entries in a row: ${4 \choose 2}, {2 \choose 1}$ or ${0 \choose 0}$, for rows in the first, second, or third block of rows. 
The presence of additional zeros in the matrix is due to $C_0$ not being MDS. 
The sets $A_0 = \{0,2\}$ and $A_0 = \{1,4\}$ are not information sets for $C_0$. 
For each of the remaining $A_0 \in J(5,2)$, $B_0(A_0)$ is an information set for $C$.
\end{example}

For $t=0$, and given a basis $g = \{ g_0, g_1, \ldots, g_{n-1} \}$ for $F^n$, the proof of Lemma \ref{L:jgc2} assigns to each $L \in J(n,v)$ a unique matrix $M$ and vector $\pi(M)$.
The vectors $\pi(M)$ form an invertible $N \times N$ matrix $\Lambda(g)$. The assignment of $\Lambda(g)$ to an ordered list of vectors $g$ is functorial. That is, if $g$ and $h$ are two ordered lists of vectors and both are interpreted as $n \times n$ matrices, with product $gh$ as $n \times n$ matrices, then $\Lambda(gh) = \Lambda(g) \Lambda(h)$ as product of $N \times N$ matrices. The functoriality property says that the determinant of
a $v \times v$ minor in $gh$ can be expressed in terms of determinants of $v \times v$ minors in $g$ and $h$. Both properties are classical and have short proofs. 

\begin{lemma}[Cauchy-Binet formula]
For a $v \times n$ matrix $A$ and $n \times v$ matrix $B$,
\[
\det(AB) = \pi(A) \cdot \pi(B^T).
\]
\end{lemma}
\begin{proof}
Compare determinants in the
$(v+n) \times (v+n)$ matrix product
\[ 
\left( \begin{array}{c|c}
I_v &A \\ \midrule
0 &I_n
\end{array}
\right) 
\left( \begin{array}{c|c}
0 &-A \\ \midrule
B &I_n
\end{array}
\right)  =
\left( \begin{array}{c|c}
AB &0 \\ \midrule
B &I_n
\end{array}
\right).
\]
using a cofactor expansion over the minors $-A_L$ of $-A$, for $L \in J(n,v)$, for the second matrix.
{\small \begin{align*}
\det \left( \begin{array}{c|c}
0 &-A \\ \midrule
B &I_n
\end{array}
\right) 
=~& \sum_{L}  \det \left( \begin{array}{c|c|c}
0 &-A_L &0 \\ \midrule
B_L &0 &0 \\ \midrule
0   &0 &I_{n-v}
\end{array}
\right) \\ 
=~ 
&\sum_{L} \det (A_L) \det (B_L) = \pi(A) \cdot \pi(B^T).
\end{align*}}
In the first equality, the blocks $-A_L$ and $B_L$ can be assumed to be in the given positions after replacing the matrix with a conjugate, i.e. by applying the same permutations to both rows and columns.  
\end{proof}

\begin{proposition} \label{P:CD}
Let $D_0$ be the dual code of $C_0$. The dual code $D$ of $C$ is generated by vectors $\pi(M)$, for all $v \times n$ matrices $M$ with at least 
$v + 1 - t$ rows in $D_0$. 
\end{proposition}

\begin{proof}
For generators $\pi(A) \in C$ and $\pi(B) \in D$, the $v \times v$ matrix $AB^T$ has a $t \times (v+1-t)$ all zero submatrix. The extended $t \times v$ submatrix
of $AB^T$ is of rank at most $t-1$ and thus $\det ( AB^T ) = 0.$ With the lemma, the generators are orthogonal.
\end{proof}


\begin{lemma}[Functoriality of exterior algebras] \label{L:func}
The assignment $g \mapsto  \Lambda(g)$, that maps a $n \times n$ matrix $g$ to a $N \times N$ matrix $\Lambda(g)$, satisfies
\[
\Lambda(gh) = \Lambda(g) \Lambda(h).
\]
\end{lemma}

\begin{proof} Evaluate entries of $\Lambda(gh)$ using the Cauchy-Binet formula.
\end{proof}

The functorial property $\Lambda(gh) = \Lambda(g) \Lambda(h)$ does not depend on the chosen orderering for the rows and columns in $\Lambda(g)$, i.e.,
on the ordering of the vertices in $J(n,v)$. However, the matrix $\Lambda(g)$, and in particular its shape, depends on the chosen ordering.
The ordering that we use to define Johnson graph codes is in general different form the lexicographic ordering and depends on $k$. Let $E_0 = \{ 0,\ldots,k-1\}$.  A vertex $L'$ preceeds $L$ if $\lvert L' \cap E_0 \rvert > \lvert L \cap E_0 \rvert.$ In case of intersections of equal size we order $L'$ and $L$ lexicographically. The partitioning of $J(n,v)$ into shells $S_r(E_0)$, for $r=0,1,\ldots,R,$
gives the matrix $\Lambda(g)$ a block structure, such that shells appear in the order $S_0(E_0), S_1(E_0), \ldots, S_R(E_0),$ and such that rows and columns within the same shell are ordered lexicographically. We call the modified order the k-lexicographic order, with notation $L' \leq_k L.$ It is a total  order on the vertices of $J(n,v)$. 

\begin{lemma}  \label{L:upp}
For an upper triangular matrix $g$, the matrix $\Lambda(g)$ is upper triangular whenever the vertex ordering on $J(n,v)$ is a refinement of the Bruhat order,
defined by the rule $(\beta_0,\ldots, \beta_{v-1}) \leq (\alpha_0,\ldots,\alpha_{v-1})$ if $\beta_i \leq \alpha_i$, for all $i$. Both the lexicographic and the $k$-lexicographic order on $J(n,v)$ have this property.
\end{lemma}

\begin{proof}
The entry $\Lambda(g)_{\beta,\alpha} = \det ( g_{ij} : i \in \beta, j \in \alpha )$ is nonzero only if $\beta_i \leq \alpha_i$ for all $i$, i.e. only if
$\beta \leq \alpha$ in the Bruhat order. Clearly this implies $\beta \leq \alpha$ in the lexicographic order as well as $|\beta \cap E_0| \geq |\alpha \cap E_0|$.
And therefore $\beta \leq_k \alpha$.
\end{proof}

A square matrix has all its pivots on the main diagonal if it is full rank and if it reduces to echelon form without row permutations. 
Clearly this is the case if and only if the matrix has a $LU$ factorization.

\begin{proposition} \label{P:LU}
Let $\Lambda(g)$ be defined with an ordering of rows and columns that refines the Bruhat order. If the matrix $g$ has a $LU$ decomposition then so has
the matrix $\Lambda(g)$.
\end{proposition}

\begin{proof}
If $g = LU$ then by functoriality $\Lambda(g) = \Lambda(L) \Lambda(U)$. For the given order, $\Lambda(L)$ is lower triangular and $\Lambda(U)$ is upper triangular by Lemma \ref{L:upp}
\end{proof}

Let $g = \{ g_0, g_1, \ldots, g_n \}$ be a basis for $F^n$. We say that $g$ is in block form, with respect to a code $C_0$ and a choice of information set $A_0$ for $C_0$, if the coordinates are ordered such that $A_0 = \{ 0,1,\ldots, k-1 \}$, if $C_0$ is the span of rows $\{ g_0, g_1, \ldots, g_{k-1} \}$, and if the remaining rows are zero in the positions $A_0$. 
\[ 
\text{block form:}~g =  \left( \begin{array}{c|c}
G_0 &G \\ \midrule
0 &G_1
\end{array}
\right), \qquad \text{systematic form:}~g =  \left( \begin{array}{c|c}
I &G \\ \midrule
0 &- I
\end{array}
\right).
\]
A matrix in block form can be reduced, using a combination of column permutations within bocks of columns and row operations within blocks of rows, to a block form with $G_0 = I$ and $G_1 = -I.$ In the reduced systematic form, $g^2 = I_n$. Thus, by Lemma \ref{L:func}, also $\Lambda(g)^2 = I_N$, independent of the ordering of vertices in $J(n,v).$ 

\begin{lemma}
For a matrix $g$ in block form, and for a $k-$lexicographic ordering of the vertices in $J(n,v)$, 
the matrix $\Lambda(g)$ is an upper triangular block matrix. Moreover, if $g$ is in systematic form then the
diagonal blocks are identity matrices up to sign. 
\end{lemma}

\begin{proof}
As in the proof of Lemma \ref{L:upp}, $\Lambda(g)_{\beta,\alpha}$ is nonzero only if $\lvert \beta \cap E_0 \rvert \geq \lvert \alpha \cap E_0 \rvert$.
The blocks $I$ and $-I$ in $g$ imply that the entry $\Lambda(g)_{\beta,\alpha}$ is nonzero for $\lvert \beta \cap E_0 \rvert = \lvert \alpha \cap E_0 \rvert$ only if $\beta = \alpha.$ 
\end{proof}

\section{Dual constructions} \label{S:dual}

In the previous section we defined codes of length $N = {n \choose v}$ with coordinates indexed by the vertices of the graph $J(n,v)$.
Codes are generated by vectors
\[
\pi ( M ) = ( \det( M_L ) : L \in J(n,v) ),
\]
for $v \times n$ matrices $M$. A vertex $L \in J(n,v)$ has complement $L^c \in J(n,n-v).$ As pointed out earlier replacing $L$ with $L^c$ amounts to a relabeling of vertices for the same graph. With the relabeling, the matrices in the code construction are replaced by $n-v \times n$ matrices. We describe
the relation between the two different cases of the same construction and include a formal verification of their equivalence. In particular, for the code $C$ and its dual code $D$, we find an equivalent pair of dual codes $C'$ and $D'$.
 
For a layer $L$ with complement $L^c$ reorder the columns in a $n \times n$ identiy matrix $I_n$ as
$\left( \begin{array}{c|c}
I_L &I_{L^c} \end{array} \right)$ and write $\sign(L) = \det( I_L \vert I_{L^c}).$
The coordinatization vector $\pi(M)$ is defined in terms of $v \times v$ determinants.
\[
\pi(M)_L = \det ( M_L ) = \det ( M I_L ).
\]
A different way to coordinatize the matrix $M$ is by extending $M$ with unit vectors $e_i$, for $i \in L^c$, 
and then taking the determinant of a $n \times n$ matrix. 
\[
\pi'(M)_L = \det \left( \begin{array}{c}
M \\ \midrule {I_{L^c}}^T \end{array} \right).
\]
The two approaches are equivalent and give vectors that are equal up to sign in each coordinate. From
\[
\left( \begin{array}{c}
M \\ \midrule {I_{L^c}}^T \end{array} \right) 
\left( \begin{array}{c}
I_L~ \vert ~I_{L^c} \end{array} \right) = \left( \begin{array}{c|c}
M_L &M_{L^c} \\ \midrule 0 &I_{n-v} \end{array} \right)
\]
it follows that 
\begin{equation} \label{eq:pi'}
\pi'(M)_L = \sign(L) \pi(M)_L. 
\end{equation}
The function $\sign(L) = \det ( I_L~ \vert ~I_{L^c} )$ depends only on $\sum_{i \in L} i$~mod~$2$ and is normalized such that 
$\sign(L) = 1$ for $L = \{ 0,1,\ldots, v-1 \}$. 

\begin{definition}
Given a code $C$ with coordinates indexed by $J(n,v)$, define the signed version $C'$ of $C$ as the equivalent code obtained by changing codewords in each of the coordinates $L \in J(n,v)$ by a factor $\sign(L)$.
\end{definition}

For a code $C$ generated by vectors $\pi(M)$, the code $C'$ is the code generated by vectors $\pi'(M)$, for the same matrices $M$. Let $D$ be the dual code of $C$ with signed version $D'$. We use the following lemma to describe $D'$ directly in terms of $C$, for the Johnson graph code $C$ in Definition \ref{def:code}. 

\begin{lemma} \label{L3}
Let $A$ be a $v \times n$ matrix and let $B$ be a $(n-v) \times n$ matrix such that the row spaces of $A$ and $B$ have nonzero intersection.
Then 
\[
\sum_L \pi'(A)_L \pi(B)_{L^c} = \sum_L \sign(L) \det (A_L) \det (B_{L^c}) = 0.
\]
\end{lemma}

\begin{proof} The first equality uses (\ref{eq:pi'}). The expression in the middle is the cofactor expansion, over the minors $A_L$ of $A$, for $L \in J(n,v)$,
for the determinant 
\[
\det \left( \begin{array}{c} A \\ \midrule B \end{array} \right) = 0.
\]
\end{proof}

Let $C_0$ be a code of length $n$. The code $C$ in Definition \ref{def:code} is generated by vectors $\pi(M)$, for all $v \times n$ matrices $M$ with at least $t$ rows in $C_0$.

\begin{proposition} \label{P:CD'}
The signed version $D'$ of the dual code $D$ of $C$ is generated by vectors $\pi(M)$, for all $(n-v) \times n$ matrices $M$ with at least 
$\dim C_0 + 1 - t$ rows in $C_0$. 
\end{proposition}

\begin{proof}
Let $\pi(A)$ and $\pi(B)$ be generators, such that $A$ has at least $t$ rows in $C_0$ and $B$ has at least $\dim C_0 + 1 -t$ rows in $C_0$.
We may assume that both $A$ and $B$ have independent rows, for otherwise $\pi(A)=0$ or $\pi(B)=0$.
Thus the row spaces of $A$ and $B$ have a nonzero intersection inside $C_0$.
With the lemma, the signed vector $\pi'(A)$ and the vector $\pi(B)$ are orthogonal.
Equivalently, the vector $\pi(A)$ and the vector $\pi'(B)$ are orthogonal. Thus $C$ and $D'$ are orthogonal spaces. Moreover, having complementary dimensions, they are dual codes.
\end{proof}

\begin{table} 
\begin{center}
$\begin{array}{lccccccc}
\toprule
&\quad &~~~C_0~~~ &\#~\text{rows in}~C_0 \geq &\quad &~~~D_0~~~ &\#~\text{rows in}~D_0 \geq \\ \midrule
J(n,v) & &C &t  & &D &v+1-t \\
J(n,n-v) & &D' &\dim C_0 +1-t & &C' &\dim D_0-v+t \\
\bottomrule
\end{array}$
\end{center}
\bigskip
\caption{Code construction for the codes $C, D, D'$ and $C'$.} \label{T:cons}
\end{table}

Table \ref{T:cons} summarizes the choice of generators $\pi(M)$ for the codes $C$, $D$, $D'$~and~$C'$ using the construction in Definition \ref{def:code}.
Proposition~\ref{P:CD} relates constructions in the same row and Proposition~\ref{P:CD'} relates constructions in the same column.
The codes $C$ and $D$, and therefore also the equivalent codes $C'$ and $D'$, are different from the trivial codes $0$~and~$F^N$ only if 
$0 < t \leq \dim C_0$ and $0 < v+1-t \leq \dim D_0.$

For comparison, Lemma \ref{L:jgc2} uses generators $\pi(M)$ for $C$ with $M$ of the form
\[
M = \left( \begin{array}{c|c|c|c}
O &I &X &Y \\ \midrule
O &O &I &O 
\end{array} \right).
\]
The codes $C$ and $D$ are nontrivial only if the column block sizes, from left to right, satisfy  
\[
\dim C_0 -t \geq 0, ~t > 0, ~v-t \geq 0, ~\dim D_0 -v+t > 0.
\]

The construction for $D$ in Table \ref{T:cons} is given by Proposition \ref{P:CD} and the construction for $D'$ by Proposition \ref{P:CD'}. The construction for $C'$ uses the two propositions in combination, either by considering $C'$ as dual of $D'$, or by considering $C'$ as signed dual of $D$. We include a direct proof for the construction of $C'$.
It is based on the following lemma.

\begin{lemma} \label{L:CC'}
Let $M$ and $N$ be full rank matrices of sizes $v \times n$ and $(n-v) \times n$, respectively, with orthogonal row spaces.
Then $\pi(N)$ is equal to $\pi'(M)$ up to a scalar factor. 
\end{lemma}

\begin{proof}
A $v \times v$ minor in $M$ is singular if and only if the complimentary $(n-v) \times (n-v)$ minor in $N$ is singular. Thus $\pi(N)$ and $\pi'(M)$ have zeros in the same positions.
For a pair of invertible minors $M_1$ and $M_2$ in $M$, with columns in $L_1$ and $L_2$, respectively, denote the complementary minors in $N$ by $N_1$ and $N_2$.
We need to show that $\sign(L_1) \det ( M_1 ) \det ( N_2 ) = \sign(L_2) \det ( M_2 ) \det ( N_1 ).$  It suffices to prove the special case where $L_1$ and $L_2$ are adjacent. 
The claim is not affected by row operations on $M$ or on $N$, or by a permutation of the columns in $M$ and $N$. Without loss of generality, we may
therefore assume that $L_1 = \{ 0,1,\ldots,v-1 \},$ $L_2 = \{ 1,2,\ldots,v \}$, and that
\[
M = \left( \begin{array}{c|c}
I_{v} &X \end{array}
\right), ~~~
N =  \left( \begin{array}{c|c}
-X^T &I_{n-v}  \end{array}
\right).
\]
So that $M_1$ and $N_1$ are identity matrices. The claim reduces to $\det ( N_2 ) = \sign(L_2) \det ( M_2 )$ and this is true for $\det( N_2) = - X_{1,1},$ $\sign(L_2) = (-1)^v$ and 
$\det(M_2) = (-1)^{v-1} X_{1,1}.$
\end{proof}

\begin{proposition} \label{P:CC'}
The signed version $C'$ of the code $C$ is generated by vectors $\pi(M)$, for all $(n-v) \times n$ matrices $M$ with at least 
$\dim D_0 -v + t$ rows in $D_0$. 
\end{proposition}

\begin{proof} We give a proof that uses Lemma \ref{L:CC'}.
For dual codes $C_0$ and $D_0$, with information sets in the leading and trailing positions, respectively,
define a pair of $n \times n$ kernel matrices
\[
g =  \left( \begin{array}{c|c}
I_{k} &B \\ \midrule
0 &- I_{n-k}
\end{array}
\right), ~~~
h =  \left( \begin{array}{c|c}
-I_k &0 \\ \midrule
-B^T &I_{n-k}
\end{array}
\right).
\]
Such that rows $\{ g_1, \ldots, g_k \}$ generate $C_0$ and rows $\{ h_{k+1}, \ldots, h_n \}$ generate $D_0$.
Generators $\pi(M)$ for $C$ are chosen with matrices $M$ that contain $v$ rows in $g$ with at least $t$ rows in $\{ g_1, \ldots, g_k \}$.
The rows in the complementary positions in $h$ form a matrix $N$ that has at least $\dim D_0-v+t$ rows in $\{ h_{k+1}, \ldots, h_n \}$.
The matrices $M$ and $N$ have orthogonal row spaces. By Lemma \ref{L:CC'}, $\pi(N)$ is proportional to $\pi'(M)$. It follows that the
vectors $\pi(N)$ generate $C'$.
\end{proof}

A weaker claim for the proposition is that the construction yields a Johnson graph code with the same parameters as $C$ and the same information sets. As Johnson graph code, $C$ is of type $JGC(n,v,k,r)$ for $r = \min ( v, \dim C_0) - t$. The constructed code is of type $JGC(n,n-v,n-k,s)$ for $s = \min( n-v, \dim D_0) - (\dim D_0 -v +t) =$ $\min (\dim C_0,v) - t = r$.  Equality of the information sets, up to the graph isomorphism between $J(n,v)$ and $J(n,n-v)$, follows from 
\[
L \in S_r(A) \subset J(n,v) ~\Leftrightarrow~ r = \min ( \lvert L \cap A^c \rvert, \lvert A \cap L^c \rvert )  ~\Leftrightarrow~ L^c \in S_r(A^c) \subset J(n,n-v).
\]

Theorem \ref{T:hsparse} gives a refinement of Lemma \ref{L:jgc3}. As in (\ref{Rd1d2}), let $d_1 = \min( v, n-k )$ and $d_2 = \min( k, n-v).$ And let $R = \min ( d_1, d_2 ).$
For an $n \times n$ matrix 
\[
h =  \left( \begin{array}{c|c}
I_{n-k} &B \\ \midrule
0 &- I_{k}
\end{array}
\right),
\]
the $N \times N$ matrix $\Lambda(h)$ with rows of the form $\pi(M)$ is such that the leading $|B_r(I_0)|$ rows, for $I_0 = \{ 0,1,\ldots,n-k-1 \}$, span $JGC(n,v,n-k,r)$, for $0 \leq r \leq R.$
The rows of $\Lambda(h)$ partition into shells $J(n,v) = \cup S_i(I_0)$, $0 \leq i \leq R.$ For $A_0 = \{ 0,1,\ldots,n-k-1 \}$, the columns of $\Lambda(h)$ partition into shells $J(n,v) = \cup S_j(A_0)$, $0 \leq j \leq R.$ The partition of rows and columns gives the matrix $\Lambda(h)$ a block structure.

\begin{theorem} \label{T:hsparse}
The $i-$th block of rows in $\Lambda(h)$ contains
\begin{equation} \label{eq:i}
| S_i(I_0) | = {n-k \choose d_1-i}{k \choose d_2-i}
\end{equation}
rows. The $j-$th segment in a row of the $i-$th block of rows in $\Lambda(h)$ contains at most 
\begin{equation} \label{eq:ij}
{d_1-i \choose j-i}{d_2-i \choose j-i}
\end{equation}
nonzero entries. The total number of nonzero entries in a row of the $i-$th block is
\begin{equation} \label{eq:rowi}
 {d_1 + d_2 - 2i \choose R-i}.
\end{equation}
\end{theorem}
 
\begin{proof} A row $\pi(M)$ in $\Lambda(h)$ belongs to the $i-$th block if
\[
M = \left( \begin{array}{c|c|c|c}
O &I &X &Y \\ \midrule
O &O &I &O 
\end{array} \right)
\]
has column block sizes, from left to right, 
\[
\dim D_0 -d_1+i, ~d_1-i, ~v-d_1+i, ~\dim C_0 -v+d_1-i.
\]
The size of the last block can be written $k - v + \min (v, n-k) -i = \min(k, n-v ) -i = d_2 -i.$
This proves the number of choices for $M$ in (\ref{eq:i}). A layer $L \in S_j(A_0)$ with $\det (M_L) \neq 0$
has number of coordinates in each block, from left to right, 
\[
0, ~d_1-j, ~v-d_1+i, ~i-j.
\]
This proves the number of choices for $L$ in (\ref{eq:ij}). Summation of (\ref{eq:ij}) over $i \leq j \leq R$ yields (\ref{eq:rowi}).
\end{proof}

\section{Construction using Reed-Solomon codes} \label{S:rs}

We consider the construction of Johnson graph codes $JGC(n,v,k,r)$ in Definition \ref{def:code} for the special case that the MDS code $C_0$ is a Reed-Solomon code. 
In general, a code of length $n$ and dimension $k$ is MDS if a codeword is uniquely determined by any $k$ of its $n$ coordinates. 
For distinct field elements $\alpha_0, \alpha_1, \ldots, \alpha_{n-1} \in F$, and for $f \in F[x],$ let
\[
ev(f) = (f(\alpha_0), f(\alpha_1), \ldots, f(\alpha_{n-1})).
\]
For $0 \leq k \leq n$, 
a Reed-Solomon code of dimension $k$ is defined as the space of vectors $\{ ev(f) : \deg f < k \}$. A codeword $ev(f)$ is uniquely determined by any $k$ of its coordinates and thus the code is MDS. 
The vectors $ev(x^i),$ for $i = 0,1, \ldots, n-1,$ form a basis for $F^n.$ They form a $n \times n$ Vandermonde matrix~$g$ with rows $g_i = ev(x^i)$ such that the leading $k$ rows of $g$ span a Reed-Solomon code of dimension~$k$, for any $0 \leq k \leq n$. 

In the previous section we used the matrix $g$ in systematic form
to describe relations among a Johnson graph code $C$, its dual code $D$, and the equivalent codes $C'$ and $D'$. Theorem \ref{T:hsparse} uses the dual matrix $h$ in systematic form to obtain sparse parity check vectors for efficient erasure correction. In this section we consider different properties and interpretations that occur when $g$ is a matrix of Vandermonde type, including connections with principal subresultants, monomial symmetric functions, Schubert codes, Chinese remainder based codes, and product matrix codes. When we row reduce the matrix $g$ it will in general be to block form rather than systematic form.

For the special case of Reed-Solomon codes, the construction in Definition \ref{def:code} assigns to a subset $I \in J(n,v)$ a coordinatization vector 
$\pi(M)$, for the matrix $M$ with rows $\{ ev(x^i) \}_{i \in I}$. The vector $\pi(M)$ has coordinates,
for $L \in J(n,v)$,
\[
\pi(M)_L 
~=~ \det (\,\alpha_j^i : i \in I, j \in L \,) =: \det(x^I; \alpha_L).
\] 

By Proposition\ref{prop:jgc}, 
the vectors
\[
\{ \; ( \text{det}(x^I; \alpha_L) )  :  L \in J(n,v) ) \; : \; I \cap \{ 0 , 1, \ldots, k-1 \} \geq t \; \}
\]
form a basis for a Johnson graph code $C$ of type $JGC(n,v,k,r)$, where $r = \min(v,k)-t.$ 
By Proposition \ref{P:CD}, 
the vectors 
\[
\{ \; ( \text{det}(x^I; \alpha_L) :  L \in J(n,v) ) \; : \; I \cap \{ 0 , 1, \ldots, n-k-1 \} \geq v+1-t \; \}
\]
form a basis for a dual Johnson graph code $D$ of type $JGC(n,v,n-k,R-r-1)$.

\begin{example} \label{ex:5221b}
The code $C = JGC(5,2,2,1)$ and its dual $D$, both defined with $J(5,2)$, are generated by vectors 
\[
\pi(M) = (\, \det \begin{pmatrix} 
\alpha_{j_0}^{i_0} &\alpha_{j_1}^{i_0}  \\
\alpha_{j_0}^{i_1} &\alpha_{j_1}^{i_1} 
\end{pmatrix} \; : \; 0 \leq j_0 < j_1 \leq 4 \, ). 
\]
\begin{align*}
C = JGC(5,2,2,1) ~&:~ 0 \leq i_0 < i_1 \leq 4,~ i_0 \leq 1. \\
D = JGC(5,2,3,0) ~&:~ 0 \leq i_0 < i_1 \leq 4,~ i_1 \leq 2.
\end{align*}
See also Example \ref{ex:5221}. The equivalent codes $C'$ and $D'$, both defined with $J(5,3)$, are generated by vectors 
\[
\pi(M) = (\, \det \begin{pmatrix} 
\alpha_{j_0}^{i_0} &\alpha_{j_1}^{i_0} & \alpha_{j_2}^{i_0} \\
\alpha_{j_0}^{i_1} &\alpha_{j_1}^{i_1} & \alpha_{j_2}^{i_1} \\
\alpha_{j_0}^{i_2} &\alpha_{j_1}^{i_2} & \alpha_{j_2}^{i_2} 
\end{pmatrix} \; : \; 0 \leq j_0 < j_1 < j_2 \leq 4 \, ). 
\]
\begin{align*}
C' = JGC(5,3,3,1) ~&:~ 0 \leq i_0 < i_1 < i_2 \leq 4,~ i_1 \leq 2. \\
D' = JGC(5,3,2,0) ~&:~ 0 \leq i_0 < i_1 < i_2 \leq 4,~ i_1 \leq 1.
\end{align*}
For $C'$ we use Proposition \ref{P:CC'} and for $D'$ we use Proposition \ref{P:CD'}.
The codes $C$ and $D$ use the same matrix $g$ when $\{ \alpha_0, \alpha_1, \ldots, \alpha_{n-1} \} = F$. In general, 
the dual code $D_0$ of a Reed-Solomon code $C_0$ is equivalent to a Reed-Solomon code.
For the general case, when the $n$ elements form a subset $\{ \alpha_0, \alpha_1, \ldots, \alpha_{n-1} \} \subset F$, let $p(z) = (z-\alpha_0)(z-\alpha_1)\cdots(z-\alpha_{n-1}).$ If $C$ is defined with $g$ then $D$ is defined with $g \Delta$, 
for $\Delta = - \mathrm{diag} (p'(\alpha_0), p'(\alpha_1), \ldots, p'(\alpha_{n-1}))^{-1}$.
\end{example}

Let $E = \{ 0,1,\ldots,n-1 \}.$ The $n \times n$ Vandermonde matrix $g$ has rows $\{ ev(x^i) \}_{i \in E}$.
Let  $f_i(x) = \prod_{j < i} (x-\alpha_j),$ for $i \in E$. The rows $\{ ev(f_i) \}_{i \in E}$ describe the matrix $g$ in row reduced 
upper triangular form. A matrix $g$ in row reduced block form depends on a choice of information set $A \subset E$ for $C_0$. 
After reordering columns we may assume $A = \{ 0,1,\ldots,k-1 \}$.
Let
\[
 h_i(x) = \begin{cases} x^i,   &0 \leq i \leq k-1. \\ f_k(x) x^{i-k}, &k \leq i \leq n-1. \end{cases}
\]
The rows $\{ ev(h_i) \}_{i \in E}$ describe the matrix $g$ in row reduced block form. 

\begin{remark}
The two reduced forms for $g$, the upper triangular form and the block form, are obtained with lower triangular row operations 
and the three matrices share the same row spaces for a set of $k$ leading rows, for any $0 \leq k \leq n.$ 
\end{remark}

For $g = \{ ev(f_i) \}_{i \in E}$ in upper triangular form, the matrix  $\Lambda(g)$ has entries
\[
\Lambda(g)_{I,L} = 
\det ( f_i(\alpha_j) : i \in I, j  \in L ) =: \det ( f_I ; \alpha_L ).
\]
For $g = \{ ev(h_i) \}_{i \in E}$ in block form, the entries are
\[
\Lambda(g)_{I,L} = 
\det ( h_i(\alpha_j) : i \in I, j  \in L ) =: \det ( h_I ; \alpha_L ).
\]

The matrix $g$ in block form depends on a choice of information set $A \in J(n,k)$. The matrix entry $\det ( h_I ; \alpha_L )$ 
further depends on a row index $I \in J(n,v)$ and a column index $L \in J(n,v)$. In Appendix \ref{S:AppB}, in Proposition \ref{P:sh} and the remark that follows it, 
we show that as a function of $A$ and $L$, for a given $I$, the entry $\det ( h_I ; \alpha_L )$ can be interpreted, in general up to a triangular transformation,
 as a generalized principal subresultant of the two polynomials $p(x) = \prod_{j \in L} (x-\alpha_j)$ and $q(x) = \prod_{j < k} (x-\alpha_j).$ 
 In particular it can be expressed as a linear combination of determinants of Sylvester type matrices for $p(x)$ and $q(x).$ 

\begin{lemma} \label{L:Ispec}
Let
\[
I = \{ 0,1,\ldots,t-1 \} \cup \{ k,k+1,\ldots,k+v-t-1 \}.
\]
That is, $I$ is minimal in $J(n,v)$ such that $I \cap \{ 0,1,\ldots,k-1 \} = t$. For all $L \in J(n,v),$
$\det ( f_I ; \alpha_L ) = \det ( h_I ; \alpha_L )$, and the upper triangular form for $g$ and the block form for $g$ 
share the same $I$-th row in $\Lambda(g).$ 
\end{lemma}

\begin{proof}
\begin{align*}
\langle \, f_i : i \in I  \rangle ~=~ 
&\langle \, f_i : 0 \leq i \leq t-1 \, \rangle +  \langle \, f_i : k \leq i \leq k+v-t-1 \, \rangle  \\
 ~=~ &\langle \, x^i : 0 \leq i \leq t-1  \, \rangle +  \langle \, f_k x^{i-k} :  k \leq i \leq k+v-t-1 \, \rangle =  \langle \, h_i : i \in I  \rangle.  
\end{align*}
\end{proof}

For the choice of $I$ in the lemma, $\det ( h_I ; \alpha_L ) = 0$ if and only if the $t$-th principal subresultant is zero for polynomials $q(x) = f_k(x)$ and $p(x) = \prod_{j \in L} (x-\alpha_j).$ 
We include an example and refer for further details to the appendix.

\begin{example} \label{ex:034a}
Let $I=\{ 0,3,4 \}$, i.e., the case $t=1, k=3, v=3$ in the lemma.
\[
\det (h_I ; x_0,x_1,x_2 ) = \det \left( \begin{array}{lll} 
1 &1 &1 \\
q(x_0) &q(x_1) &q(x_2) \\
q(x_0) x_0 &q(x_1) x_1 &q(x_2) x_2 
\end{array} \right).
\]
Where $q(x) = f_3(x) = (x-\alpha_0)(x-\alpha_1)(x-\alpha_2)$. Let $p(x) = (x-x_0)(x-x_1)(x-x_2)$ and let $d(x) = \gcd (p(x),q(x)).$ 
For $\deg d(x) > 1$, e.g., for $q(x_0) = q(x_1) = 0$, the matrix is singular. For $\deg d(x) = 1$, e.g., for $q(x_0)=0, q(x_1), q(x_2) \neq 0,$ the matrix is regular. For $q(x) = \sum_i q_i x^{k-i}$, $p(x) = \sum_i p_i x^{v-i}$,
\[
\det ( h_I ; x_0,x_1,x_2) = 0 ~\Leftrightarrow~   \det \left( \begin{array}{cccc} p_0 &p_1 &p_2 &p_3 \\ 0 &p_0 &p_1 &p_2 \\ \midrule q_0 &q_1 &q_2 &q_3 \\ 0 &q_0 &q_1 &q_2 \end{array} \right) = 0.
\]
The right side is the first principal subresultant for $p(x)$ and $q(x).$
\end{example}

Let $\Lambda(g)$ be the $\binom{n}{v} \times \binom{n}{v}$ matrix $\{ \det(x^{L_1} ; \alpha_{L_2}) \}_{L_1, L_2 \in J(n,v) }$. 
As possible orderings for the rows and columns in $\Lambda(g)$ we consider the lexicographic and $k$-lexicographic orderings for the vertices of $J(n,v)$.
By Lemma \ref{L:upp} and Proposition \ref{P:LU}, the principal minors for $\Lambda(g)$ are invertible for either of the two orderings.
We prove, for the case of a Vandermonde matrix $g$, that for each of the two orderings the principal minors of $\Lambda(g)$ factor completely into a product of binomial differences $\alpha_i-\alpha_j.$

\begin{lemma} \label{L:ord1}
Let $L = (\ell_0, \ell_1, \dots, \ell_{v-1}) \in J(n,v).$ The number of vertices $L' \in J(n,v)$ that are adjacent to $L$ and that precede $L$ in the
lexicographic order is $| L | :=  \sum_{i =0}^{v-1} (\ell_i - i)$.
\end{lemma}

\begin{proof}
To replace $\ell_i \in L$ with $\ell_j \not \in L$ such that $\ell_j < \ell_i$ there are $\ell_i - i$ choices.
The total number of choices for $L'$ is therefore $\sum_{i =0}^{v-1} (\ell_i - i).$
\end{proof}

The $k$-lexicographic order was defined before Lemma \ref{L:upp}. Let $E_0 = \{ 0,1,\ldots,k-1 \}$. Recall that $L' \leq_k L$ if $\lvert L' \cap E_0 \rvert > \lvert L \cap E_0 \rvert,$ or, in case of equality, if $L' \leq L$ lexicographically. 

\begin{lemma}  \label{L:ord2}
For adjacent vertices $L', L \in J(n,v)$, 
\[
L' < L  ~\Leftrightarrow~  L' <_k L ~\Leftrightarrow~ |L'| < |L|.
\]
\end{lemma}

\begin{proof}
Let $a = L' \backslash L$ and $b = L \backslash L'$. Clearly, $L' < L$ if and only if $a < b$ if and only if $|L'| < |L|.$
Moreover, $L' <_k L$ if and only if ($a < k \leq b$ or $a < b \leq k-1$ or $k \leq a < b$) if and only if $a < b$.
\end{proof}

\begin{proposition}
Let $g$ be a Vandermonde matrix. Let rows and columns in $\Lambda(g)$ share the same ordering in which row $L'$ precedes row $L$ whenever $L', L \in J(n,v)$ are adjacent  
and $|L'| < |L|.$ Then, for every $L \in J(n,v)$, the determinant of the submatrix $\Lambda(g)_{\leq L} = \Lambda(g)_{L_1, L_2 \leq L}$ 
factors as a product of binomial differences $\alpha_i - \alpha_j$, $i \neq j$. 
\end{proposition}

\begin{proof}
We prove by induction to $L$ that $d_{\leq L} = \det \Lambda(g)_{\leq L}$ factors completely, as polynomial in $\alpha_0, \alpha_1, \ldots, \alpha_{n-1},$ into factors $\alpha_i - \alpha_j$. In general factors appear with multiplicities. For the base case $L_0 = \{ 0,\ldots,v-1 \}$, the scalar $\Lambda(g)_{\leq L_0} =  \det(x^{L_0}; \alpha_{L_0})$ is a Vandermonde determinant of degree $d_0 = \sum_{i=0}^{v-1} i$. For the inductive case, observe that each new column $L$ contributes two types of linear factors to $d_{\leq L}$. There is a first contribution of $d_0$ factors because every entry in column $L$ is divisible by $\det(x^{L_0}; \alpha_L)$ of degree $d_0$. Secondly, 
there is a contribution of one linear factor $\alpha_i-\alpha_j$ for each $L' < L$ that is adjacent to~$L$.
By Lemma \ref{L:ord1} and Lemma \ref{L:ord2} the number of such $L'$ is $|L|.$
As $\Lambda(g)_{< L}$ is enlarged to $\Lambda(g)_{\leq L}$, both the degree of the determinant and its number of linear factors increases by $|L|+d_0$. 
\end{proof}

The proposition applies to all principal minors of the matrix $\Lambda(g)$. In particular, for the full matrix $\Lambda(g)$ it follows that 
$\det \Lambda(g) = (\det g)^m$, for $m = {n-1 \choose v-1}.$ 

Let $I_0 = (0,1, \ldots, v-1)$ and let $\pi_0 = ( \det (x^{I_0}; \alpha_L) : L \in J(n,v) )$ be a vector of $v \times v$ Vandermonde determinants.
Replacing the determinant 
$\det(x^I ; \alpha_L)$ in a generator $\pi(M) = ( \det (x^{I}; \alpha_L) : L \in J(n,v) )$ for a Johnson graph code with its Schur polynomial $\det (x^I; \alpha_L) / \det (x^{I_0}; \alpha_L)$ gives an equivalent set of rescaled generating vectors $\pi(M) \Delta_0$,
for the diagonal matrix $\Delta_0 = \mathrm{diag}\,( \pi_0 )^{-1}.$ This replaces the code with an equivalent code with the same properties
but normalized such that the leading generator for the code is the allone vector. 

Schur polynomials are homogeneous symmetric polynomials with unique expressions on a basis
of monomial symmetric functions \cite[Corollary 7.10.6]{Stanley}. We point out and illustrate with an example that the triangular matrix that transforms Schur polynomials into monomial symmetric functions is in general not compatible with the $k$-lexicographic order that we use for the generators of a Johnson graph code. Thus, in general, replacing Schur functions with the leading monomial symmetric functions in their expansion will change the code. It can be shown however that these different codes are again Johnson graph codes.

\begin{example} \label{ex:6331}
Consider the Johnson graph code $JGC(6,3,3,1)$. Let $A = \{ 0, 1, 2 \}.$ 
The Schur polynomials of degree $4$ and $5$, and their expansion as sums of monomial symmetric functions are
\begin{align*}
\begin{array}{rrrrrrrrrrr}
I = \{ 0, 2, 5 \} &r = 1 & &s_{310} = &m_{310} &+ & m_{220} &+ &2 m_{211}, \\ \noalign{\smallskip}
      \{ 0, 3, 4 \} &2 & &s_{220} = &                  &  & m_{220} &+ &m_{211}, \\  \noalign{\smallskip}
      \{ 1, 2, 4 \} &1 & &s_{211} = &                  &  &               &  &m_{211}. 
\end{array}
\\
\begin{array}{rrrrrrrrr}
I = \{ 0, 3, 5 \} &r = 2 & &s_{320} = &m_{320} &+ & m_{311} &+  &2 m_{221}, \\
      \{ 1, 2, 5 \} & 1 & &s_{311} = &                  &  & m_{311} &+ &m_{221}, \\
      \{ 1, 3, 4 \} & 2 & &s_{221} = &                  &  &               &  &m_{221}. 
\end{array}
\end{align*}
The Schur polynomials $s_{310}, s_{211}, s_{311}$ have $I = 025, 124, 125 \in S_1(012)$ and the
corresponding generators $\pi(M)$ are generators for $JGC(6,3,3,1).$ They can not be expressed as linear combinations 
of the corresponding monomial symmetric functions $m_{310}, m_{211}, m_{311}$. The triangular transformation that
relates Schur polynomials and monomial symmetric functions of the same degree in general involves Schur polynomials from
different shells. 
\end{example}

We mention three types of codes that share common properties with Johnson graph codes: Grassmann codes and Schubert codes, Polynomial Chinese remainder codes, and product matrix MSR codes. 

For $v \leq n$, rational points on the Grassmann variety $G_{v,n}$ over a finite field $F$ correspond to $v$ dimensional $F$-linear subspaces of $F^n$. The variety has a smooth embedding, using Pl\"ucker coordinates, into projective space $\mathbb{P}^{N-1}$, for $N = {n \choose v}.$ The Grassmann code
is the row space of the $N \times K$ matrix whose columns are the images in $\mathbb{P}^{N-1}$ of rational points $P_1, P_2, \ldots, P_K \in G_{v,n}$. 
Over a field $F$ of size $q$, $K \leq \# G_{v,n}(F) = {n \choose v}_q$. Let $g = \{ g_1, \ldots, g_n \}$ be a basis for $F^n$. For $\alpha \in J(n,v)$, 
write the elements in $\alpha$ in increasing order $1 \leq \alpha_1 < \cdots < \alpha_v \leq n$. The Schubert variety $\Omega_\alpha \subset G_{v,n}$
is the subvariety of subspaces $W$ whose intersection with the flag generated by $g$ reaches dimension $i$ after at most $\alpha_i$ steps. 
\[
\Omega_\alpha = \{\, W \in G_{v,n} \,:\, \dim (W \cap \langle g_1, \ldots, g_{\alpha_i} \rangle) \geq i, ~ \text{for $i = 1,\ldots,v.$} \,\}.
\]
A Schubert code is defined by projecting the $N \times K$ matrix onto a $N \times K_\alpha$ submatrix with columns in $\Omega_\alpha$. In general this reduces the rank of the matrix to $k_\alpha$, and yields a Schubert code of dimension $k_\alpha$ and length $K_\alpha$. 
With hindsight, our construction of Johnson graph codes matches the column space of the $N \times K_\alpha$ matrix, for the maximal $\alpha$ in the Bruhat order with $\alpha_t = k.$ The Johnson graph code, as row space of a $K_\alpha \times N$ matrix of rank $k_\alpha$, has dimension $k_\alpha$ and length $N$. In our description, we find generators for the code by taking the top $k_\alpha = \lvert B_r(\{0,1,\ldots,k-1\}\rvert$ rows in a $N \times N$ matrix $\Lambda(g)$.
With the above we can interpret information sets for the Johnson graph code as projections of the Pl\"ucker embedding of a Schubert variety that
separate points.  

\bigskip

We give a different construction for codes $J(n,v,v,1)$. The proof for the general case follows after the example.
 
\begin{example}
Example \ref{ex:5221b} describes a code $C = JGC(5,2,2,1)$ and the equivalent code $C' = JGC(5,3,3,1)$.
The first code uses an embedding of the pair $\{ x,y \} \subset \{ \alpha_0,\alpha_1,\ldots,\alpha_4 \}$ with Pl\"ucker coordinates $\det_I(x,y)$,
$I \in \{ 01, 02, 03, 04, 12,13,14 \}$. The second code uses an embedding of the triplet $\{ x,y,z \} \subset \{ \alpha_0,\alpha_1,\ldots,\alpha_4 \}$ with Pl\"ucker coordinates $\det_I(x,y,z)$, $I \in \{ 012, 013, 014, 023, 024, 123, 124 \}$. A different embedding that also yields a code $JGC(5,2,2,1)$ is
as vector of coefficients for the polynomial
\[
(1+xt+x^2t^2+x^3t^3)(1+yt+y^2t^2+y^3t^3)    \qquad (n=5, v=2).
\]
Similarly the coefficients for the polynomial
\[
(1+xt+x^2t^2)(1+yt+y^2t^2)(1+zt+z^2t^2)     \qquad (n=5,v=3).
\]
give an embedding for $\{ x,y,z \}$ that yields a code $JGC(5,3,3,1).$
\end{example}

\begin{proposition}
Let $v \leq n$. Let $f_0, f_1, \ldots, f_{n-1} \in F[x]$ be $n$ polynomials of degree $n-v$, such that
any two polynomials are relative prime and any $n-v+1$ polynomials are linearly independent. For each $L \in J(n,v)$, let $f_L = \prod_{i \in L} f_i.$ Then $\deg f_L = v(n-v).$
For any given $A \in J(n,v)$, the collection $\{ f_L : \lvert L \cap A \rvert \geq v-1 \}$ forms  a basis for $F[x]_{\leq v(n-v)}.$
\end{proposition}

\begin{proof}
The size of the collection $\{ f_L \}$ is $1+v(n-v),$ which is the dimension of $F[x]_{\leq v(n-v)}.$ It suffices to prove that the $f_L$ in the collection span
 $F[x]_{\leq v(n-v)}.$ For any $i \in A$, the collection contains $n-v+1$ multiples $(f_A / f_i) \cdot f_j,$ for $j=i$ or $j \not \in A.$ Since the $f_j$
 span $F[x]_{\leq n-v}$, the collection contains all multiples of $f_A / f_i$, for all $i \in A.$ Since $f_j$ and $f_i$ are relative prime, the collection spans
 all multiples of $f_A / f_i f_j$, for all pairs $i,j \in A$. With induction it follows that the collection spans all multiples of $1$, i.e., spans $F[x]_{\leq v(n-v)}.$   
\end{proof}

\begin{corollary}
A special case of the proposition is $f_i = 1 + \alpha_i x + \cdots + \alpha_i^{n-v} x^{n-v}$, for distinct $\alpha_0, \alpha_1, \ldots, \alpha_{n-1} \in F$,
such that $\gcd(n-v+1,|F|-1)=1$.
\end{corollary}

A special case of a Johnson graph code is the following. Let $\alpha_0, \alpha_1, \ldots, \alpha_{n-1} \in F$ be distinct field elements. Let $S$ be a symmetric matrix of size $k-1 \times k-1$, with associated symmetric form $f(x,y) = (1,x,\ldots,x^{k-2}) S (1,y,\ldots,y^{k-2})^T$. Label vertices $\{ i,j \} \in J(n,2)$ with $f(\alpha_i,\alpha_j)$. Then $S$, and thus $f(\alpha_i,\alpha_j)$ for all $\{ i,j \} \in J(n,2),$ is uniquely determined by the set of values $f(\alpha_i, \alpha_j)$, $i,j \in A$, for any $k$-subset $A$ of $\{ 0,1,\ldots,n-1 \}.$ Thus the code is of type
$JGC(n,2,k,0).$ In \cite{Rashmi+11}, two copies of the code are combined to form an MSR regenerating code. Let the second copy be defined with a matrix $T$ and symmetric form $g(x,y)$.
For each $i \in \{ 0,1,\ldots,n-1 \}$, node $i$ stores the polynomial $f(\alpha_i,y)+\alpha_i g(\alpha_i,y),$ which is a polynomial of degree $k-2$ in the single variable $y$.
When data is collected from $k$ nodes $A$, any two accessed nodes $i,j \in A$ can recover $f(\alpha_i,\alpha_j)$ and $g(\alpha_i,\alpha_j)$ from
$f(\alpha_i,\alpha_j)+\alpha_i g(\alpha_i,\alpha_j)$ and $f(\alpha_j,\alpha_i)+\alpha_j g(\alpha_j,\alpha_i)$, that is together they can decouple their stored values. 
With the decoupled values, the matrices $S$ and $T$ can be recovered as before. In this construction, the Johnson graph code is used as inner code and the coupling as outer code.
In the next section, Johnson graph codes are used as outer codes and layered codes as inner codes.   

\section{Concatenated Layered Codes} \label{S:appl}

In this section we use Johnson graph codes as a tool to concatenate layered codes
with different layer sizes into longer codes with improved download properties. In a layered code, data in different layers is independent, and data can in general only be recovered by accessing any $n-1$ out of $n$ nodes. In a concatenated code, Johnson graph codes provide relations between data in different layers that ensure that data can be recovered from fewer than $n-1$ nodes. The design of concatenated codes is concerned with (1) The amount of data that remains unrecovered when accessing layered codes in fewer than $n-1$ nodes, (2) The extra amount of data that can be recovered through the use of Johnson graph codes, and (3) How the latter can compensate for the former. For the most part in this section we only need to refer to basic features of layered codes and Johnson graph codes which will reduce the design problem to a combinatorial tiling problem. 

The main feature of layered codes that we use is a partition into shells of all layers of a given size when $k$ of the $n$ nodes are accessed. Table~\ref{tbl:48} gives a partition into shells for layers of size five when $n=8$ and $k=4$. For the same $n$ and $k$, Table~\ref{T:v5k4} gives the shell sizes for the partitions
of layers with size five or smaller.

 \begin{table}[h]
$\begin{array}{rccccccccccccccccccccccccccc} \\ \toprule  \noalign{\smallskip}
\multicolumn{2}{r}{|A \cap L|=} &4 &3 &2 &1 &0 & &\# L \\ 
 \noalign{\smallskip}
 \midrule 
 \noalign{\smallskip}
|L|=5 & &{4 \choose 4} {4 \choose 1} &{4 \choose 3} {4 \choose 2}  &{4 \choose 2} {4 \choose 3}  &{4 \choose 1} {4 \choose 4}  &  & &{8 \choose 5} \\ \noalign{\medskip}
4 & &{4 \choose 4} {4 \choose 0} &{4 \choose 3} {4 \choose 1}  &{4 \choose 2} {4 \choose 2}  &{4 \choose 1} {4 \choose 3}  &{4 \choose 0} {4 \choose 4}  & &{8 \choose 4} \\ \noalign{\medskip}
3 & & &{4 \choose 3} {4 \choose 0} &{4 \choose 2} {4 \choose 1} &{4 \choose 1} {4 \choose 2}  &{4 \choose 0} {4 \choose 3}  & &{8 \choose 3} \\ \noalign{\medskip}
2 & & & &{4 \choose 2} {4 \choose 0} &{4 \choose 1} {4 \choose 1}  &{4 \choose 0} {4 \choose 2} & &{8 \choose 2} \\ \noalign{\medskip}
1 & & & & &{4 \choose 1} {4 \choose 0} &{4 \choose 0} {4 \choose 1} & &{8 \choose 1} \\
\noalign{\smallskip} \midrule 
 \end{array}$
\medskip
\caption{Breakdown of all layers $L \subset \{ 0, 1, \ldots, 7 \}$ of a given size according to how they intersect $A$ of size $|A|=4$.}  \label{T:v5k4}
\end{table}

Symbols in a layer $L \subset \{ 0, 1, \ldots, n-1 \}$ are protected by a single parity $\sum_{i \in L} c_i = 0$. So that symbols in a layer $L$ are known to a data collector $A \subset \{ 0, 1, \ldots, n-1 \}$ with access to symbols $\{ c_i : i \in A \}$ if and only if $| A \cap L | \geq |L|-1$. 
   
\begin{definition}
Given a layered code with $n$ nodes, and upon accessing a subset of nodes $A \subset \{ 0,1,\ldots,n-1 \}$, we call a layer $L$ \emph{fully-accessed}
if all symbols in $L$ are available to the data collector. We call $L$ \emph{sufficiently-accessed} if $L$ is not fully-accessed but its symbols can be recovered from the collected symbols with the single parity check on the symbols in~$L$. A layer that can not be recovered using a combination of
collected symbols and parity checks is called \emph{under-accessed}.
\end{definition}

In Table \ref{T:v5k4}, sufficiently-accessed layers, with $|A \cap L| = |L|-1$, appear on the diagonal. Fully-accessed layers appear below the diagonal and under-accessed layers above the diagonal. A fully-accessed layer $L$ does not need the parity check $\sum_{i \in L} c_i = 0$ to recover its symbols. If we replace the check with a check of the
form $\sum_{i \in L} c_i = s$, for a symbol $s$ of our choice, then the fully-accessed layer can pass the symbol $s$ to a different under-accessed 
layer. 
The transfer of data proceeds in rounds. Each round requires a choice of three parameters $w$, $v$ and $r$. Here $w$ is the size of a fully-accessed layer providing helper data, $v$ is the size of an under-accessed layer receiving helper data, and $r$ restricts the layers receiving helper data to those with
$w \leq |A \cap L | \leq v-1-r.$ It is assumed that at the time of transfer, data has been recovered from all layers of size $w < v$ and from all layers of size $v$ with $|A \cap L | > v-1-r.$ 
A Johnson graph code $JGC(n-w,v-w,k-w,r)$ is used to transfer one symbol of helper data to each layer $L$ with $|L|=v$ and $w \leq |A \cap L | \leq v-1-r$ using a combination of helper data from sublayers of size $w$ and data available from layers with $|L|=v$ and $|A \cap L | > v-1-r.$ 
Table \ref{T:v5k4} can be used to keep track of the amount of data that needs to be transferred to under-accessed layers. The problem of full
data recovery reduces to a combinatorial tiling problem of choosing a sequence of rounds $(w_i,v_i,r_i)$ that will transfer enough helper data to under-accessed layers, in the right order, to recover all their data.  

There is a wide range of choices for the sequence $(w_i,v_i,r_i)$ to achieve full data recovery. 
There is a unique sequence if we restrict the choice to $w_i=1$ in each round. This results in concatenated codes with the same storage and bandwidth parameters as the improved layered codes in \cite{Senthoor+15}. Another choice is to set $r = v-1-w$ in each round, so that helper data from $L_w$ is transferred only
to layers $L$ with $A \cap L = L_w$. This results in codes with the same storage and bandwidth as cascade codes in \cite{Elyasi+18}. We include examples of each type that recover data in the layers of size $|L|=5$ in Table \ref{T:v5k4} under the assumption that data in smaller layers has been recovered. 

\begin{example} \label{ex:dt847}

\begin{table}[h]
$\begin{array}{rccccccccccccccccccccccccccc} \\ \toprule  \noalign{\smallskip}
\multicolumn{2}{r}{|A \cap L|=} &4 &3 &2 &1 &0 & &\text{multiplicity} \\ 
 \noalign{\smallskip}
 \midrule 
 \noalign{\smallskip}
|L|=5 & & &-6 \cdot 1  &-4 \cdot 2  &-1 \cdot 3  &  & &1\,\times \\
4 & &+1 \cdot 1 &  &-6 \cdot 1  &-4 \cdot 2  &-1 \cdot 3  & &0\,\times \\
3 & & &+1 \cdot 1 &  &-6 \cdot 1 &-4 \cdot 2 & &6\,\times \\
2 & & & &+1 \cdot 1 &  &-6 \cdot 1 & &8\,\times \\
1 & & & & &+1 \cdot 1 & & &39\,\times \\
\noalign{\smallskip} \midrule 
\multicolumn{1}{l}{\Sigma=}  & &0 &0 &0 &0 &-96           \\  
 \noalign{\smallskip} \bottomrule
 \end{array}$
\medskip
\caption{Unique multiplicities such that within each column except the last fully-accessed layers (below the diagonal) compensate under-accessed layers (above the diagonal).} \label{T3}
\end{table}

Table \ref{T3} corresponds to a concatenation of layered codes with $v=5$ ($1 \times$), $v=3$ ($6 \times$), $v=2$ ($8 \times$), and  $v=1$ ($39 \times$).
The multiplicities are unique if we restrict data transfer to transfer between layers in the same column, i.e. data from $L_w$ is transferred only
to layers $L$ with $A \cap L = L_w$. 
We describe in detail how for the given multiplicities data stored in layers of size $w < 5$ can be used to recover data in under-accessed layers of size $v=5$. Let $A = \{ 0,1,2,3 \}$. When considered as vertices in the Johnson graph $J(8,5)$, the ${8 \choose 5}=56$ layers of size $5$ divide over shells $S_0(A), S_1(A), S_2(A), S_3(A)$ of sizes $4, 24, 24, 4$, respectively (Table \ref{tbl:48}). Layers in $S_0(A)$ are sufficiently-accessed and layers in the remaining shells are under-accessed. 

We recover the data stored in the layers $S_1(A), S_2(A), S_3(A)$ recursively in three rounds. For the recovery, helper data wil be transferred from sublayers of size $w=3, w=2, w=1,$ respectively. For $A = \{ 0,1,2,3 \}$ there are multiple fully-accessed layers of size $w=3$ $(4 \times)$, $w=2$ $(6 \times)$, and $w=1$ $(4 \times)$.   

\begin{table}[h]
$\begin{array}{ccc}
\begin{array}{c} J(5,2) \\[1.5ex] (w=3, L_w = \{1,2,3\}) \end{array} &\qquad &\begin{array}{
                       c@{\hspace{3mm}}c
                       c
                       @{\hspace{2mm}}c
                       c@{\hspace{2mm}}c@{\hspace{2mm}}c
                       @{\hspace{2mm}}c
                       c@{\hspace{2mm}}c@{\hspace{2mm}}c@{\hspace{2mm}}c
                       @{\hspace{3mm}}ccccccccccccc} 
\toprule
{\#}  &   &0   & &\multicolumn{3}{c}{1, 2, 3} & &\multicolumn{4}{c}{4, 5, 6, 7} & &r  \\ 
\midrule
4 &      &\st    & &\st&\st&\st & &\st &\ds &\ds &\ds & & 0  \\
6 &      &\ds   &  &\st &\st &\st & &\st&\st &\ds &\ds & & 1  \\
\bottomrule 
\end{array} \\ \noalign{\bigskip}
\begin{array}{c} J(6,3) \\[1.5ex] (w=2, L_w = \{ 2,3\}) \end{array} &\qquad &\begin{array}{
                       c@{\hspace{3mm}}c
                       c@{\hspace{2mm}}c
                       @{\hspace{2mm}}c
                       c@{\hspace{2mm}}c
                       @{\hspace{2mm}}c
                       c@{\hspace{2mm}}c@{\hspace{2mm}}c@{\hspace{2mm}}c
                       @{\hspace{3mm}}ccccccccccccccc} 
\toprule
{\#}  & &\multicolumn{2}{c}{0, 1} & &\multicolumn{2}{c}{2,3} & &\multicolumn{4}{c}{4, 5, 6, 7} & &r  \\ 
\midrule
4 & &\st&\st & &\st &\st & &\st &\ds &\ds &\ds & & 0  \\
12 & &\ds&\st & &\st &\st & &\st&\st &\ds &\ds & & 1  \\
4 & &\ds &\ds & &\st &\st & &\st&\st&\st &\ds  & & 2 \\
\bottomrule 
\end{array}   \\ \noalign{\bigskip}
\begin{array}{c} J(7,4) \\[1.5ex] (w=1, L_w = \{ 3 \}) \end{array} &\qquad &\begin{array}{
                       c@{\hspace{3mm}}c
                       c@{\hspace{2mm}}c@{\hspace{2mm}}c
                       @{\hspace{2mm}}c
                       @{\hspace{2mm}}c
                       @{\hspace{2mm}}c
                       c@{\hspace{2mm}}c@{\hspace{2mm}}c@{\hspace{2mm}}c
                       @{\hspace{3mm}}ccccccccccccccc} 
\toprule
{\#}  & &\multicolumn{3}{c}{0, 1, 2} & &3 & &\multicolumn{4}{c}{4, 5, 6, 7} & &r  \\ 
\midrule
4 & &\st&\st&\st & &\st & &\st &\ds &\ds &\ds & & 0  \\
18 & &\ds&\st &\st & &\st & &\st&\st &\ds &\ds & & 1  \\
12 & &\ds &\ds &\st & &\st & &\st&\st&\st &\ds  & & 2 \\
1 & &\ds &\ds&\ds & &\st & &\st&\st&\st&\st  & & 3 \\
\bottomrule 
\end{array}\\ \noalign{\bigskip}
\end{array}$ 
\caption{Subgraphs of $J(8,5)$ with vertices $L \supset L_w$ and their partition into shells for $A = \{ 0,1,2,3 \}.$} \label{Tsub}
\end{table} 

(Round 1 - $w=3, v=5, r=1$) To each layer $L_w$ of size $3$ corresponds a Johnson graph $J(5,2) \subset J(8,5)$ with vertex labeling $(c_L)$. For a fully-accessed layer $L_w$, the ten layers in $J(5,2)$ belong to
$S_0(A) (4 \times), S_1(A) (6 \times)$ (Table \ref{Tsub}). The layer $L_w$ stores a vector $s = H_1 c$ of length $6$. The vector $s$ together with four available $c_L$ determines the full vector $(c_L)$. For $L \in S_1(A)$, after receiving $c_L$, the data in $L$ can be recovered. 

(Round 2 - $w=2, v=5, r=2$) To each layer $L_w$ of size $2$ corresponds a Johnson graph $J(6,3) \subset J(8,5)$ with vertex labeling $(c_L)$. For a fully-accessed layer $L_w$, the 20 layers in $J(6,3)$ belong to
$S_0(A) (4 \times), S_1(A) (12 \times), S_2(A) (4 \times)$ (Table \ref{Tsub}). The layer $L_w$ stores a vector $s = H_2 c$ of length $4$. The vector $s$ together with 16 available $c_L$ determines the full vector $(c_L)$. For $L \in S_2(A)$, recovery of its data requires two additional symbols, and we store two vectors $s$ of length $4$ at each layer $L_w$. With the two additional symbols, the data in $L$ can be recovered. 

(Round 3 - $w=1, v=5, r=3$) To each layer $L_w$ of size $1$ corresponds a Johnson graph $J(7,4) \subset J(8,5)$ with vertex labeling $(c_L)$. For a fully-accessed layer $L_w$, the 35 layers in $J(7,4)$ belong to
$S_0(A) (4 \times), S_1(A) (18 \times), S_2(A) (12 \times), S_3(A)(1 \times)$  (Table \ref{Tsub}). The layer $L_w$ stores a vector $s = H_3 c$ of length $1$. The vector $s$ together with 34 available $c_L$ determines the full vector $(c_L)$. For $L \in S_3(A)$, recovery of its data requires three additional symbols, and we store three vectors $s$ of length $1$ at each layer $L_w$. With the three additional symbols, the data in $L$ can be recovered. 
\end{example}

\begin{proposition} \label{prop:rec} Let
\[
f_\ell(t) = \frac{1}{(1-\ell t)(1+t)^\ell} = \sum_{i \geq 0} a_i t^i.
\] 
The concatenation of a single layered code with layers of size $v$ together with $a_i$ copies of layered codes
with layers of size $v-i$, for $0 < i < v$, results in a concatenated code whose data is recoverable from any $k=n-1-\ell$ nodes.
\end{proposition}

\begin{proof}
For general number of nodes $n$, top layer size $|L|=v$, and number of accessed nodes $|A|=k$, a sublayer of size $v-i$ provides $a_i$ helper symbols
to $a_{i-j} {n-k \choose j}$ layers of size $v-i+j$, each of which receives $j-1$ symbols for full data recovery, for $0 < j \leq i$. This gives the recursion 
$a_i = \sum_{0 < j \leq i} a_{i-j} {n-k \choose j} (j-1),$ with $f_{n-k-1}(t)$ as generating function.
 \end{proof}
 
We verify the parameters $M$ (total amount of data stored), $\alpha$ (number of symbols stored at a single node) and $\beta$ (number of symbols provided by each of $d$ helper nodes for repair of a failed node) for the
concatenated code obtained with the multiplicities in Table \ref{T3}. This is done in Table \ref{T4}. The columns $M_1, \alpha, \beta$ contain the parameters
of the layered codes and their sum for the given multiplicities. The 96 missing symbols in column~0 of Table~\ref{T3} are not compensated by a fully-accessed layer below the diagonal. The 96 symbols are not recoverable. In Table \ref{T4} this number is denoted by $M_0$. It is subtracted from $M_1$ to obtain the
size $M = M_1 - M_0$ of the data that is stored by the concatenated code. 

 \begin{table}[h]
$\begin{array}{rcclrccccccccccccccccc} \\ \toprule  \noalign{\smallskip}
& &\# L & &M_1 &\alpha &\beta &~  & &M_0 &M_1 - M_0&\text{multiplicity} \\ 
 \noalign{\smallskip}
 \midrule 
 \noalign{\smallskip}
|L|=5 & &56 & &224 &35 &20   & & & &224     &1\,\times \\
4 & &70 & &210 &35 &15   & & &1 \cdot 3 &207   &0\,\times \\
3 & &56 & &112 &21 &6     & & &4 \cdot 2 &104   &6\,\times  \\
2 & &28 & &28 &7 &1       & & &6 \cdot 1 &22      &8\,\times  \\
1 & &8 & &0 &1 &0         & & & &0    &39\,\times       \\  
\noalign{\smallskip} \midrule 
\multicolumn{1}{l}{\Sigma=}  & & & &1130 &256 &64         & & &96 &1024 &      \\  
 \noalign{\smallskip} \bottomrule
 \end{array}$
\medskip
\caption{Parameters $(M = M_1 - M_0, \alpha, \beta) = (1024,256,64)=(4^5,4^4,4^3)$ for a $(n,k,d)=(8,4,7)$ concatenated layered code.} \label{T4}
\end{table}

\begin{corollary} For the concatenated code in Proposition \ref{prop:rec},
\[
\alpha = [t^{v-1}]\,f_\ell(t)(1+t)^{n-1}, \qquad \beta = [t^{v-2}]\, f_\ell(t)(1+t)^{n-2}.
\]
For the special case $v=n-\ell=k+1$,
\[
M = k \alpha, \quad \alpha = (n-k)^{k}, \quad \beta = (n-k)^{k-1}.
\]
The latter parameters satisfy the MSR conditions $M = k \alpha$ and $\alpha = (d-k+1)\beta.$
\end{corollary}

\begin{table}[h]
\begin{tabular}{c}
$J(5,2) \qquad \begin{array}{ccccccccccccc} 
\toprule
{\#}  &   &r   & &\text{code} & &\text{type} \\ 
\midrule
4 &   &0  & &JGC(5,2,1,0) & &[10,4] \\
6 &   &1  & &JGC(5,2,1,1) & &[10,10] \\
\midrule
\end{array}$ \\ 
$J(6,3)  \qquad \begin{array}{ccccccccccccccc} 
4 & &0  & &JGC(6,3,2,0) & &[20,4]   \\
12 & &1  & &JGC(6,3,2,1) & &[20,16]   \\
4 & &2  & &JGC(6,3,2,2) & &[20,20]  \\
\midrule
\end{array}$   \\ 
$J(7,4) \qquad \begin{array}{ccccccccccccccc} 
4 & &0  & &JGC(7,4,3,0) & &[35,4]   \\
18 & &1  & &JGC(7,4,3,1) & &[35,22]   \\
12 & &2  & &JGC(7,4,3,2) & &[35,34]   \\
1 & &3  & &JGC(7,4,3,3) & &[35,35]   \\
\bottomrule 
\end{array}$\\ \noalign{\bigskip}
\end{tabular} 
\caption{Johnson graph codes on subgraphs of $J(8,5)$} \label{T9}
\end{table}

The data transfer in Example \ref{ex:dt847} is from sublayers of size $w=3$ in Round 1, size $w=2$ in Round 2, and size $w=1$ in Round 3. In each round it is possible to use sublayers of a smaller size as source for the transfer. Using the short notation $3-2-1$ for the given choices, the alternative choices are $3-1-1$, $2-2-1$ or $1-1-1.$ In each case, the data transfer follows Example \ref{ex:dt847} but with a different choice of Johnson graph codes. The 
different Johnson graph codes affect the parameters of the overall concatenated code. Thus the choice of the sources for the data transfer is an important aspect when designing concatenated layered codes.

Table \ref{T9} lists different Johnson graph codes on subgraphs of $J(8,5)$ and Table \ref{T10} lists how many helper symbols are available to an under-accessed layer in $J(8,5)$ for a given choice of Johnson graph code. 
Table \ref{T11} includes the choice of Johnson graph codes and their multiplicities for the choices $3-2-1$, $3-1-1$, $2-2-1$ and $1-1-1.$ Use of a Johnson graph code of type $[n,k]$ requires a layer $L_w$ to store $n-k$ symbols. Table \ref{T11} includes for each of the choices $w_1 - w_2 - w_3$ a column with the number of copies that layered codes of sizes $w_1, w_2, w_3$ contribute to the overall concatenated layered code. From these multiplicities and the parameters of the individual layered codes it is straightforward to compute the parameters of the concatenated layered code, as in Table \ref{T4}. We include details for the case $1-1-1.$

\begin{table}[h]
$\begin{array}{rccccccccccccccccccccccccccc} \\ \toprule  \noalign{\smallskip}
 &\quad  &\#L_w &|s| &\quad  &S_1(A) &S_2(A) &S_3(A)  \\
 \noalign{\smallskip}
 \midrule 
 \noalign{\smallskip}
JGC(5,2,1,0) & &4 &6 & &24 &0 &0  \\ \midrule
JGC(6,3,2,0) & &6 &16 & &72 &24 &0   \\
JGC(6,3,2,1) & &6 &4 & &0 &24 &0   \\ \midrule
JGC(7,4,3,0) & &4 &31 & &72 &48 &4  \\
JGC(7,4,3,1) & &4 &13 & &0 &48 &4  \\
JGC(7,4,3,2) & &4 &1 & &0 &0 &4 \\
\noalign{\smallskip} \midrule 
\noalign{\smallskip} 
\Sigma = & & & & &24 &48 &12 \\
\noalign{\smallskip} \bottomrule 
 \end{array}$
\medskip
\caption{Transfer by $\# L_w$ fully-accessible sublayers $L_w$ of $|s|$ helper symbols to shells $S_i(A)$.  Required total number of helper symbols
for each shell $S_i(A)$ in the last row.}  \label{T10}
\end{table}

\begin{table}[h]
$\begin{array}{rccccccccccccccccccccccccccc} \\ \toprule  \noalign{\smallskip}
 & &\multicolumn{1}{c}{3-2-1} & &\multicolumn{1}{c}{3-1-1} 
 & &\multicolumn{1}{c}{2-2-1} & &\multicolumn{1}{c}{1-1-1} \\ 
 \noalign{\smallskip}
 \midrule 
 \noalign{\smallskip}
JGC(5,2,1,0) & &1  & &1 & &-   & &- \\ \midrule
JGC(6,3,2,0) & &-  & &-  &  &1/3  & &-  \\
JGC(6,3,2,1) & &2  & &-  &  &5/3 & &-  \\ \midrule
JGC(7,4,3,0) & &- & &- &  &-  &  &1/3 \\
JGC(7,4,3,1) & &- & &1  &  &-   & &2/3   \\
JGC(7,4,3,2) & &3  & &2 & &3  & &2   \\
\noalign{\smallskip} \midrule 
 \end{array}$
\medskip
\caption{Data recovery scenarios $w_1-w_2-w_3$ with multiplicities that meet the required sum in Table \ref{T10}}  \label{T11}
\end{table}

\begin{example}
(The case $1-1-1$) In this case, all data is transferred from layers of size $v=1$. A layer of size $v=1$ stands for a single node. This represents a degenerate case outside the range $2 \leq v \leq n$. Data stored in a layer of size $v=1$ is simply helper data added to a node for the purpose of being available when a layer containing that node is under-accessed. 
Let $A = \{ 0,1,2,3 \}$ and $L_w = \{ 3 \}$. Table \ref{T9} gives the partition of the subgraph $J(7,4)$ (of layers $L \in J(8,5)$ that contain $L_w$) into
shells. Data is recovered sequentially in three rounds for layers in the shells $r=1,2,3$. The transfer of data from $L_w$ to under-accessed layers $L$ in shell $r$ uses a Johnson graph code $J(7,4,3,r)$, for $r=1,2,3$. The three codes are of type $[35,4],$ $[35,22]$, and $35,34],$ respectively. We rescale the multiplicities $1/3, 2/3, 2$ in Table \ref{T11} to integers $1, 2, 6$ (and rescale at the same time from $4$ to $12$ the data that is stored in a layer of size $5$).  
The parities to be stored in layer $L_w = \{ 3 \}$ (i.e., at the single node 3) are one vector $H_1 c_1$ of length $31$, two vectors $H_2 c_2$ of length $13$, and six vectors $H_3 c_3$ of length $1$. A total of $63$ symbols. 
The layered code of type $(8,7,7)$ has parameters $M = { 8 \choose 5 } \cdot 4, \alpha = {7 \choose 4}, \beta = {6 \choose 3}.$ After the rescaling by a factor $3$ and the addition of $63$ additional symbols to each node we obtain a concatenated layered code of type $(8,4,7)$ with
$M= 672, \alpha = 168, \beta = 60.$ 
\end{example}

\begin{table}[h]
$\begin{array}{rccccccccccccccccccccccccccc} \\ \toprule  \noalign{\smallskip}
 & &5 &4 &3 &2 &1 &\quad &M &\alpha &\beta \\
 \noalign{\smallskip}
 \midrule 
 \noalign{\smallskip}
3-2-1 & &1 &- &6 &8 &39 & &1024 &256  &64 \\ 
3-1-1 & &1 &- &6 &- &51 & &848 &212  &56 \\ 
2-2-1 & &3 &- &- &36 &9 & &1464 &366  &96 \\ 
1-1-1 & &3 &- &- &- &63 & &672 &168  &60 \\ 
\noalign{\smallskip} \midrule 
 \end{array}$
\medskip
\caption{Four different concatenated codes}  \label{T12}
\end{table}


\appendix


\section{Hamming graph codes} \label{S:HG}

Johnson graph codes, defined in Section \ref{S:Jdefn}, are codes of length ${n \choose v}$ with ${n \choose k}$ information sets of special form. For $v=1$ the codes are MDS. For $v>1$, the codes are in general not MDS. But the existence of many information sets of special form makes them suitable for applications in distributed storage. In this section we consider an analogous family of codes on Hamming graphs. Hamming graph codes are codes of length ${n^m}$ with ${n \choose k}$ information sets of special form. For $m=1$ the codes are MDS. For $m>1$, the codes are in general not MDS. We summarize the properties for Johnson graph codes and formulate the analogous properties for Hamming graph codes.

\medskip

The Johnson graph $J(n,v)$ has as vertices all $v$-subsets of $E = \{ 0,1,\ldots,n-1\}$. For a subset $A \subset E$ and for $r \geq 0$, let 
\[
B_r(A) = \{ L \in J(n,v) : |L \backslash A| \leq r ~\text{or}~ |A \backslash L| \leq r \}.
\]

(Definition \ref{d1})
A Johnson graph code $JGC(n, v, k, r)$ on the vertices of the Johnson graph $J(n, v)$ is a code of length $N = \binom{n}{v}$ and dimension 
$K = \lvert B_r( A ) \rvert$ such that, for any $k$-subset $A$ of $\{0,1,\dots, n-1\}$, $B_r(A)$ forms an information set. 

\medskip

Section \ref{S:Jdefn} further gives the following duality result.

\medskip

 (Proposition \ref{p:dualjgc})
The dual code of a Johnson graph code $JGC(n,v, k, r)$ is a code $JGC(n,v,n-k,D-r-1)$.

\medskip

Replacing the role of a Johnson graph with that of a Hamming graph we obtain a class of codes that includes binary Reed-Muller codes.
The Hamming graph $H(m,n)$ has as vertices all ordered $m$-tuples over an alphabet $I$ of size $n$. Two vertices in $H(m,n)$ are adjacent if their $m$-tuples differ in precisely one coordinate. 
For a subset $A \subset I$ and for $r \geq 0$, let 
\[
B_r(A) = \{ L \in H(m,n) : |\{ 0 \leq i \leq m-1 : L_i \not \in A \}| \leq r \}.
\]
Note that for $A$ of size $k$ and for $r=0$, $B_0(A)$ is the subset $H(m,k) \subset H(m,n)$ of all ordered $m$-tuples over $A$. 

\medskip

\begin{definition} \label{d2}
A Hamming graph code $HGC(m,n,k,r)$ on the vertices of the Hamming graph $H(m,n)$ is a code of length $N = n^m$ and dimension 
$K = \lvert B_r( A ) \rvert$ such that, for any $k$-subset $A$ of $\{0,1,\dots, n-1\}$, $B_r(A)$ forms an information set. 
\end{definition}

\begin{proposition} \label{p:dualhgc}
The dual of a Hamming graph code $HGC(m,n, k, r)$ is a code $HGC(m,n,$ $n-k,m-r-1)$.
\end{proposition}

\begin{proof}
The proof follows that of Proposition \ref{p:dualjgc} and uses that $B_r(A)^c = B_{m-r-1}(A^c).$
\end{proof}

\begin{example} (Reed-Muller codes)
Binary Reed-Muller codes $RM(r,m)$ are Hamming graph codes $HGC(m,2,1,r)$, with $n=2$ and $k=1.$
Codewords for the binary Reed-Muller codes $RM(r,m)$ are defined by evaluation of squarefree $m$-variable polynomials of degree at most $r$ in all binary $m$-tuples $\{0,1\}^m$.
The dimension $K = {m \choose 0}+{m \choose 1}+\cdots+{m \choose r}$ of the code $RM(r,m)$ is the number of squarefree $m$-variable monomials of degree at most $r$. 
Clearly $K$ equals the size of a neighborhood $B_r(a)$ of radius $r$ around a vertex $a \in H(m,2)$. Moreover, 
coordinates in the positions $B_r(a)$ form an information set for the  Reed-Muller code $RM(r,m)$. Note that our definition of Hamming graph code 
only requires that $B_r(a)$ is an information set for $a$ the allzero or the allone vector.
The dual of the code $RM(r,m)$ is the code $RM(m-1-r,m).$ 
\end{example}

The construction of Johnson graph codes in Section \ref{S:Jcons} is algebraic.

\medskip

(Definition \ref{def:code}) For given $n$ and $v$, let $C_0$ be a code of length $n$ and let $t \geq 0$. Define a code $C$ of length $N = {n \choose v}$ as the span of the vectors $\pi(M)$ for all $v \times n$ matrices $M$ with at least $t$ rows in $C_0$.

\medskip

The vector $\pi(M)$ is defined as $( \det( M_L ) : L \in J(n,v) ),$
 where $M_L$ is the full minor of $M$ with columns in the $v$-subset $L \subset \{ 0, 1, \ldots, n-1 \}$.

\medskip

Let $L = (L_0, L_1, \ldots, L_{m-1}) \in H(m,n)$ be a $m$-tuple over the alphabet $\{ 0,1,\ldots,n-1 \}$. Given a $m \times n$ matrix $M$ over a field $F$, define the vector $\tau(M)$ as
\[
\tau ( M ) = ( \prod_{i=0}^{m-1} M_{i,L_i} : L \in H(m,n) ).
\]
The vector $\pi(M)$ gives coordinates for the exterior product of rows in a matrix $M$, and the vector $\tau(M)$ gives coordinates for the tensor product of rows in a matrix $M$

\begin{definition} (Codes on Hamming graphs) \label{def:coders} For given $n$ and $m$, let $C_0$ be a code of length $n$ and let $t \geq 0$. Define a code $C$ of length $N=n^m$ as the span of the vectors $\tau(M)$ for all $m \times n$ matrices $M$ with at least $t$ rows in $C_0$.
\end{definition}

In Section \ref{S:Jcons}, Definition \ref{def:code} is followed by three lemmas. All three lemmas have analogs for the codes in Definition \ref{def:coders}.

\medskip

(Lemma \ref{L:jgc1})
The code $C$ has dimension $K =  \lvert B_r(I_0) \rvert,$
where $r = \min ( v, \dim C_0) -t$ and $I_0 \subset \{ 0, 1, \ldots, n-1 \}$ is of size $\dim C_0$.

\medskip

(Lemma \ref{L:jgc2})
For $C$ as in Definition \ref{def:code}, and for an information set $A_0 \subset  \{ 0, 1, \ldots, n-1 \}$ for $C_0$, the set
\[
 \{ \; L \in J(n,v) ~|~ \lvert L \cap A_0 \rvert \geq t \; \} = B_r(A_0)
\]
is an information set for $C$.

\medskip

(Lemma \ref{L:jgc3}) For $L \in S_r(A_0),$ the constructed codewords $\pi(M)$ in the proof of Lemma~\ref{L:jgc2} have as special property that the coordinate
$\pi(M_L)$ is the unique nonzero coordinate in the positions $B_r(A_0)$. Moreover the weight of $\pi(M)$ is at most 
 \[
{ 2t+n-\dim C_0-v \choose t }. 
\]

\medskip

For codes on Hamming graphs (Definition \ref{def:coders}) we have the following three lemmas.

\begin{lemma} \label{L:dim}  (In analogy with Lemma \ref{L:jgc1})
The code $C$ has dimension $K =  \lvert B_r((I_0)^m) \rvert,$
where $r = m-t$ and $I_0 \subset \{ 0, 1, \ldots, n-1 \}$ is of size $\dim C_0$.
\end{lemma}

\begin{proof} Let $g = \{ g_0, g_1, \ldots, g_{n-1} \}$ be a basis for $F^n$ and let
$\{ g_i : i \in I_0 \}$ be a basis for $C_0$.
The vectors $\tau(M)$, for matrices $M$ with rows $( g_i : i \in I )$ such that $d(I,(I_0)^m) \leq m-t$, form a basis for $C$. 
Thus $\dim C = \lvert B_r((I_0)^m) \rvert$, for $r = m-t.$ 
\end{proof} 

\begin{lemma} \label{L:15h} (In analogy with Lemma \ref{L:jgc2})
For $C$ as in the definition and for an information set $A_0 \subset  \{ 0, 1, \ldots, n-1 \}$ for $C_0$, the set
$B_r((A_0)^m)$ is an information set for $C$, for $r = m-t.$
\end{lemma}

\begin{proof} The size of $B_r((A_0)^m)$ equals the dimension of $C$. It therefore suffices to prove the independence of the coordinates in~$B_r((A_0)^m)$.
For each $L = (j_1, j_2, \ldots, j_m ) \in S_r((A_0)^m)$, we construct a codeword $c = \tau(M)$ with $c_{L} = 1$ and with $c_{L'} = 0$ for $L' \in S_r((A_0)^m)
 \backslash L.$
 For $j_i \in A_0$ include as row $i$ for $M$ the unique
codeword in $C_0$ that is $1$ in $j_i$ and $0$ in $A_0 \backslash j_i.$ For $j_i  \not \in A_0$ include as row $i$ for $M$ the unit vector that is $1$ in $j_i$ and $0$ elsewhere. The resulting matrix $M$ has rows of the form 
\begin{align*}
( \begin{array}{c|c|c|c} 0 &e &x &y \end{array} ), ~~~\text{for $j_i \in A_0. $}\\
( \begin{array}{c|c|c|c} 0 &0 &e &0 \end{array} ), ~~~\text{for $j_i \not \in A_0.$} 
\end{align*}
The first two blocks of columns in $M$ correspond to $A_0 \subset \{ 0, 1 \ldots, n-1 \}$ and the second and third block to $\{ j_1, j_2, \ldots, j_m \} \subset \{ 0, 1, \ldots, n-1 \} $. The total size of these two blocks is at most $m$ and is less than $m$ when $L$ has repeated symbols. In that case $M$ has repeated rows. The vectors $e$ are unit vectors such that each vector is $1$ in position $j_i$.
Clearly, $\tau(M)_L = \prod_{j_i \in L} M_{i,j_i} = 1.$ Moreover $\tau(M)_{L'} = 0$ for $L' \neq L$ with $d(L',(A_0)^m) \leq
d(L,(A_0)^m).$ The constructed codewords $\tau(M)$, for $L \in S_r(A_0), \ldots, S_1(A_0), S_0(A_0)$, in that order, form an invertible triangular matrix in the positions $B_r((A_0)^m)$. And thus $B_r((A_0)^m)$ is an information set for $C$.  
\end{proof}

\begin{lemma} (In analogy with Lemma \ref{L:jgc3})
For $L \in S_r(A_0),$ the constructed codewords $\tau(M)$ in the proof of Lemma~\ref{L:15h} have as special property that the coordinate
$\tau(M_L)$ is the unique nonzero coordinate in the positions $B_r(A_0)$. Moreover the weight of $\tau(M)$ is at most $(n-k+1)^t$.
\end{lemma}

\begin{proof} The weight of $\tau(M)$ is the product of the weights of the rows in $M$. The matrix $M$ has $t$ rows of weight at most $n-k+1$. The remaining rows are of weight one.
\end{proof}

The binary Reed-Muller code $RM(r,m)$ has words of minimum weight $2^{m-r}$. The code is defined with $n=2$, $k=1$, and $t = m-r$.

\begin{table}[h]
{\small
\[
 \begin{array}{cc cccl | rccccccl | rccc}
00 &~~ &1&\cdot&\cdot&\cdot~  &~1&1&\cdot&\cdot &1&\cdot&1&\cdot~ &~1&1&1&1 \\
01 &~~ &\cdot&1&\cdot&\cdot  &0&1&\cdot&\cdot &\cdot&1&\cdot&1 &0&1&0&1 \\
10 &~~ &\cdot&\cdot&1&\cdot  &\cdot&\cdot&1&1 &0&\cdot&1&\cdot &0&0&1&1 \\
11 &~~ &\cdot&\cdot&\cdot&1  &\cdot&\cdot&0&1 &\cdot&0&\cdot&1 &0&0&0&1 \\ \cmidrule{3-18}
02 &~~ &\cdot&\cdot&\cdot&\cdot  &1&\cdot&\cdot&\cdot &\cdot&\cdot&\cdot&\cdot &1&\cdot&1&\cdot \\
03 &~~ &\cdot&\cdot&\cdot&\cdot  &\cdot&1&\cdot&\cdot &\cdot&\cdot&\cdot&\cdot &\cdot&1&\cdot&1 \\
12 &~~ &\cdot&\cdot&\cdot&\cdot  &\cdot&\cdot&1&\cdot &\cdot&\cdot&\cdot&\cdot &0&\cdot&1&\cdot \\
13 &~~ &\cdot&\cdot&\cdot&\cdot  &\cdot&\cdot&\cdot&1 &\cdot&\cdot&\cdot&\cdot &\cdot&0&\cdot&1 \\ 
20 &~~ &\cdot&\cdot&\cdot&\cdot  &\cdot&\cdot&\cdot&\cdot &1&\cdot&\cdot&\cdot &1&1&\cdot&\cdot \\
21 &~~ &\cdot&\cdot&\cdot&\cdot  &\cdot&\cdot&\cdot&\cdot &\cdot&1&\cdot&\cdot &0&1&\cdot&\cdot \\
30 &~~ &\cdot&\cdot&\cdot&\cdot  &\cdot&\cdot&\cdot&\cdot &\cdot&\cdot&1&\cdot &\cdot&\cdot&1&1 \\
31 &~~ &\cdot&\cdot&\cdot&\cdot  &\cdot&\cdot&\cdot&\cdot &\cdot&\cdot&\cdot&1 &\cdot&\cdot&0&1 \\ \cmidrule{3-18}
22 &~~ &\cdot&\cdot&\cdot&\cdot  &\cdot&\cdot&\cdot&\cdot &\cdot&\cdot&\cdot&\cdot &1&\cdot&\cdot&\cdot \\
23 &~~ &\cdot&\cdot&\cdot&\cdot  &\cdot&\cdot&\cdot&\cdot &\cdot&\cdot&\cdot&\cdot &\cdot&1&\cdot&\cdot \\
32 &~~ &\cdot&\cdot&\cdot&\cdot  &\cdot&\cdot&\cdot&\cdot &\cdot&\cdot&\cdot&\cdot &\cdot&\cdot&1&\cdot \\
33 &~~ &\cdot&\cdot&\cdot&\cdot  &\cdot&\cdot&\cdot&\cdot &\cdot&\cdot&\cdot&\cdot &\cdot&\cdot&\cdot&1  
\end{array}
\]
}
\caption{Matrix with row and column parition $S_0 | S_1 | S_2$.} \label{T:ham16}
\end{table}

\begin{example}
For the Hamming graph $H(m=2, n=4)$, let 
\[
g = \left( \begin{array}{c}
g_0 \\ g_1 \\ g_2 \\ g_3 \end{array} \right) = \left( \begin{array}{ccccc}
1 &0 &1 &1 \\ 0 &1 &0 &1 \\ 0 &0 &1 &0 \\  0 &0 &0 &1 \end{array} \right).
\]

Let $k=2$, $I_0 = \{ 0,1 \}$, $C_0 = \text{span}(g_0, g_1)$. For $t=1$, $r=m-t=1$, and the parameters define a code $C$ of
type $HGC(2,4,2,1).$ 

The code $C$ is spanned by the first two blocks of rows in Table \ref{T:ham16}. Matrix entries $\cdot$ are $0$ independent of the choice of the code $C_0$. 
The block structure of the matrix follows the partition
\begin{align*}
S_0(I_0 \times I_0)  ~=~ &\{ 00, 01, 10, 11 \}, \\
S_1(I_0 \times I_0) ~=~ &\{ 02, 03, 12, 13, 20, 21, 30, 31 \}, \\  
S_2(I_0 \times I_0)  ~=~ &\{ 22, 23, 32, 33 \}.
\end{align*}
The matrix is a conjugate of $g \otimes g$ and differs from it by a permutation of rows and columns. 
For $A_0 = \{ 0 , 1 \}$, the layer $L = 21 \in S_1(A_0 \times A_0)$ defines a matrix
\[
M = \left( \begin{array}{cccc} 0 &0 &1 &0 \\ 0 &1 &0 &1 \end{array} \right)
\]
with codeword $\tau(M) = (0000 | 00000100 | 0100)$. It is nonzero
in the positions $21$ and $23$, with $L=21$ the unique nonzero position in $B_1(A_0 \times A_0).$ 
Among the six sets $B_1(A_0 \times A_0)$, for $A_0 \in \{ 01, 02, 03, 12, 13, 23 \}$, all except the one for $A_0 = 02$ are information sets. 
\end{example}

\begin{proposition} \label{P:CDrs} (In analogy with Proposition \ref{P:CD})
Let $D_0$ be the dual code of $C_0$. The dual code $D$ of $C$ is generated by vectors $\tau(M)$, for all $m \times n$ matrices $M$ with at least 
$m + 1 - t$ rows in $D_0$. 
\end{proposition}

\begin{proof}
Generators $\tau(A)$ for $C$ and $\tau(B)$ for $D$ have inner product 
\begin{align*}
\tau(A) \cdot \tau(B) =  \sum_L \left( \prod_{i=0}^{v-1} A_{i,L_i} B_{i,L_i} \right) 
=  \prod_{i=0}^{v-1} \left( \sum_{j=0}^{n-1}  A_{i,j} B_{i,j} \right).
\end{align*}
At least one of the factors in the product is zero and the two generators are orthogonal.
\end{proof}

\section{Bezout equation and subresultants} \label{S:AppB}

For $p(x), q(x) \in F[x]$, let $d(x) = \gcd ( p(x),q(x) )$ and $\delta = \deg d(x).$ The Bezout equation asks for coefficients $a(x), b(x) \in F[x]$ such that
\[
a(x) p(x) + b(x) q(x) = d(x).
\]
For $\deg p(x) = v$ and $\deg q(x) = k$, let
\[
\phi : F[x]_{<k} \times F[x]_{<v} \longrightarrow F[x]_{<k+v}, \quad (a,b) \mapsto ap+bq. 
\]
The map is linear and is represented on a basis of monomials by the $(k+v) \times (k+v)$ Sylvester matrix $\Sigma(p,q)$ for $p(x)$ and $q(x)$.  The map is invertible if and only if $p(x)$ and $q(x)$ are relatively prime. 
\begin{equation} \label{eq:spq0}
\det \Sigma(p,q) = 0 ~\Leftrightarrow~ \delta >0. 
\end{equation}
For $0 \leq i \leq \min (v,k)$, the map 
\[
\phi_i : F[x]_{<k-i} \times F[x]_{<v-i} \longrightarrow F[x]_{<k+v-i}, \quad (a,b) \mapsto ap+bq,
\]
is represented by a $(k+v-2i) \times (k+v-i)$ matrix. Let $\Sigma_i(p,q)$ be the square submatrix that
corresponds to the projection of $ap+bq$ onto its $k+v-2i$ leading coefficients. It holds that
\begin{equation} \label{eq:spqi}
\det \Sigma_i(p,q) = 0 ~\text{if $\delta >i$,} \qquad \det \Sigma_i(p,q) \neq 0 ~\text{if $\delta = i$.} 
\end{equation}
This includes (\ref{eq:spq0}) as the special case $i=0$. We first restate (\ref{eq:spqi}) in terms of a size $k+v-i$ matrix
$\Phi_i(p,q)$ and will then generalize the first part of (\ref{eq:spqi}) in Lemma~\ref{L:genI}. 
The expression for $\det \Sigma_i(p,q)$ in the coefficients of $p(x)$ and $q(x)$ is known as the $i$-th principal subresultant of
$p(x)$ and $q(x)$. 

\begin{example} \label{ex:034b}
For $p(x) = \sum_i p_i x^{3-i},$ $q(x) = \sum_i q_i x^{3-i},$ $a(x) = \sum_i a_i x^{1-i}$, $b(x) = \sum_i a_i x^{1-i}$, 
the polynomial $d(x)  = a(x)p(x)+b(x)q(x) = \sum_i d_i x^{4-i}$ has leading coefficients  
\[
 \left( \begin{array}{cccc}d_0 &d_1 &d_2 &d_3 \end{array} \right) = \left( \begin{array}{cc|cc}a_0 &a_1 &b_0 &b_1 \end{array} \right) 
 \left( \begin{array}{cccc} p_0 &p_1 &p_2 &p_3 \\ 0 &p_0 &p_1 &p_2 \\ \midrule q_0 &q_1 &q_2 &q_3 \\ 0 &q_0 &q_1 &q_2 \end{array} \right).
\]
The determinant of the $4 \times 4$ matrix gives the first principal subresultant of $p(x)$ and $q(x).$
\end{example}

Let $F[x]_{<k+v-i} = V_0 \oplus V_1$, for
\[
V_0 = \langle x^j : 0 \leq j < i \rangle,  
\quad V_1 = \langle x^j : i \leq j < k+v-i \rangle.
\]
Then $\Sigma_i(p,q)$ represents the composite function $\sigma_i = \pi_1 \circ \phi_i : F[x]_{<k-i} \times F[x]_{<v-i}  \longrightarrow  V_1.$ 
An alternative to projection on the leading $k+v-2i$ coefficients of $ap+bq$ is to mask the trailing $i$ coefficients. This is achieved by extending $\phi_i$ to a map 
\[
\phi'_i : F[x]_{<i} \times F[x]_{<k-i} \times F[x]_{<v-i} \longrightarrow F[x]_{<k+v-i}, \quad (a_0,a,b) \mapsto a_0+ap+bq.
\]
The map is represented by a size $k+v-i$ square matrix $\Phi_i(p,q)$ with determinant 
\[
\det \Phi_i(p,q) = \det \Sigma_i(p,q).
\]
It is easy to verify that $\phi'_i$ is not surjective if $\delta >i$ and that
$\phi'_i$ is injective if $\delta = i,$ in accordance with (\ref{eq:spqi}).

We generalize the map $\sigma_i$ and the first part of criterion (\ref{eq:spqi}). For $\deg p(x) = v$ and $\deg q(x) = k$, let $E=\{ 0,1,\ldots,k-1 \},$ and let $I$ be a subset of $v$ nonnegative integers.
Let $m$ be minimal such that $E \cup I \subset \{ 0,1,\ldots,m-1 \},$ and let
\[
\phi_I : F[x]_{<m-v} \times F[x]_{<m-k} \longrightarrow F[x]_{<m}, \quad (a,b) \mapsto ap+bq.
\]
The domain is of dimension $2m-k-v.$ Let $F[x]_{<m} = V_0 \oplus V_1 \oplus V_2$, for
\[
V_0 = \langle x^i : i \in E \cap I \rangle,  
\quad V_1 = \langle x^i : i \in (E \cup I) \backslash (E \cap I) \rangle,  \quad V_2 = \langle x^i : i \not \in E \cup I \rangle.
\]
Then $\dim V_0 + \dim V_1 + \dim V_2 = m$ and $\dim V_2 = m-k-v + \dim V_0.$ 
In particular $\dim V_1 + 2 \dim V_2 = 2m-k-v.$ Define a map
\[
\sigma_I ~:~ F[x]_{<m-v} \times F[x]_{<m-k} \longrightarrow V_1 \times V_2 \times V_2, \quad (a,b) \mapsto
( \pi_1 (ap+bq), \pi_2(ap), \pi_2(bq) ).
\]
For $I = \{ 0,1,\ldots,i-1 \} \cup \{ k,k+1, k+v-i-1 \}$, with $0 \leq i \leq \min (v,k)$, the map reduces to
the special case $\sigma_I = \sigma_i$. In that case $V_2 = 0.$ To derive properties for the general case we extend the map $\phi_I$ to a map $\phi'_I,$
\[
\phi'_I : V_0 \times F[x]_{<m-v} \times F[x]_{<m-k} \longrightarrow F[x]_{<m} \times V_2, \quad (a_0,a,b) \mapsto (a_0+ap+\pi_1(bq), \pi_2(bq)).
\]
Note that $\deg bq \geq k$ implies $bq \in V_1 + V_2$ and thus $a_0 + ap + \pi_1(bq)$ $+\pi_2(bq) =
a_0 + ap +bq.$ 
 
\begin{lemma} \label{L:genI}
The map $\phi'_I$ is singular if $\delta > \lvert I \cap E \rvert$. 
\end{lemma}
\begin{proof}
In the first coordinate, $a_0$ has support in $E \cap I$ and $ap+\pi_1(bq)$ has support in $\{ \delta, \ldots, m-1 \}$. The map is
surjective only if $\{ 0,1,\ldots,\delta-1 \} \subset E \cap I,$ hence only if $\delta \leq | E \cap I |.$ 
\end{proof}

Let
\begin{align*}
\sigma'_I ~:~ &V_0 \times  F[x]_{<m-v} \times F[x]_{<m-k} \longrightarrow V_0 \times V_1 \times V_2 \times V_2. \\
&~~(a_0,a,b) \mapsto (a_0 + \pi_0(ap),\pi_1 (ap+bq), \pi_2(ap), \pi_2(bq) ). 
\end{align*}
Then $\sigma'_I = ((\pi_0,\pi_1,\pi_2),id) \circ \phi'_I$ differs from $\phi'_I$ by a bijective map 
$(\pi_0,\pi_1,\pi_2) : F[x]_{<m} \rightarrow V_0 \oplus V_1 \oplus V_2$ and thus Lemma \ref{L:genI} applies to $\sigma'_I.$ 
Moreover, the matrices representing the two maps on a basis of monomials are the same and thus they share
the same determinant. The matrix representing $\sigma_I$ has different size but the same determinant as $\sigma'_I.$

\begin{corollary} \label{C:genI}
The map $\sigma_I$ is singular if $\delta > \lvert I \cap E \rvert$. 
\end{corollary}  

Let $\Sigma_I(p,q)$ denote the size $2m-k-v$ matrix that represents the linear map $\sigma_I$.

\begin{example} \label{ex:si}
For $E = \{ 0,1,2 \}$ and  $I = \{ 1,4 \}$, $m=5$. The matrix $\Sigma_I$ includes columns that project $ab+bq$
onto coordinates in $(E \cup I) \backslash (E \cap I) = \{ 0, 2, 4 \}$ and columns that project $ap$ and $bq$ 
onto coordinates in $\{ 0,1,\ldots,m \} \backslash (E \cup I) = \{ 3 \}.$ 
\[
\Sigma_I (p,q) = 
\bordermatrix{
&0  &2 &3 &3  &4 \cr \noalign{\smallskip}
&p_2 &p_0 &\cdot &\cdot &\cdot \cr
&\cdot &p_1 &p_0 &\cdot &\cdot  \cr
&\cdot &p_2 &p_1 &\cdot &p_0 \cr \cline{2-6}
&q_3  &q_1 &\cdot &q_0 &\cdot  \cr
&\cdot &q_2 &\cdot &q_1 &q_0 \cr
}
\]
\end{example}    

Let $\alpha_0, \alpha_1, \ldots, \alpha_{n-1} \in F$, and let
$q(x) = (x-\alpha_0)(x-\alpha_1)\cdots(x-\alpha_{k-1}).$ 
To relate $\det \Sigma_I (p,q)$ to the determinants $\det ( h_I ; \alpha_L )$ in Section \ref{S:rs}, we replace $\sigma_I$ with a map $\eta_I$ that is defined in terms of the zeros of $p(x).$ Let $p(x)  = \prod_{j \in L} (x-\alpha_j)$ have $v$ distinct zeros,
indexed by $L \in J(n,v)$. Recall that $\det ( h_I ; \alpha_L ) = \det ( h_i(\alpha_j) : i \in I, j  \in L )$,
where
\[
 h_i(x) = \begin{cases} x^i,   &0 \leq i \leq k-1. \\ q(x) x^{i-k}, &k \leq i \leq n-1, \end{cases}.
 \]

Let  $ev_L : F[x] \longrightarrow F^v, ~~f \mapsto ( f(\alpha_j) : j \in L).$ Define a map
\[
\eta_I : ~:~ V_0 \times F[x]_{<m-k} \longrightarrow  F^v \times V_2, \quad (a_0,b) \mapsto ( ev_L(a_0 + \pi_1(bq)),~~\pi_2(bq)  ).
\]
Since $\dim V_0 + m-k = v + \dim V_2$, the map $\eta_I$ is represented by a square matrix.
The map $\eta_I$ extends to a map
\begin{align*}
\eta'_I ~:~ &V_0 \times F[x]_{<m-v} \times F[x]_{<m-k} \longrightarrow F[x]_{\geq v} \times F^v \times V_2. \\
&~~(a_0,a,b) \mapsto (\pi_{\geq v} (a_0 + ap + \pi_1(bq)), ev_L(a_0+\pi_1(bq)), \pi_2(bq)).
\end{align*}
The determinants for $\eta'_I$ and $\eta_I$ differ by a factor $p_0^{m-v}$, $\det \eta'_I = p_0^{m-v} \det \eta_I.$
The map $\eta'_I = ((\pi_{\geq v},ev_L),id) \circ \phi'_I$ differs from $\phi'_I$ by a bijective map and Lemma \ref{L:genI} applies to $\eta'_I.$

\begin{corollary} The map $\eta_I$ is singular if $\delta > \lvert I \cap E \rvert$.
\end{corollary}
  
\begin{proposition} \label{P:sh}
\[
\det (ev_L) \det \sigma_I = p_0^{m-v} \det \eta_I.
\]
\end{proposition}

\begin{proof}
Combine $\det \eta'_I = \det (ev_L) \det \phi'_I = \det (ev_L) \det \sigma_I$ and $\det \eta'_I = p_0^{m-v} \det \eta_I.$
\end{proof}
   
Note that $\det (ev_L )$ is the Vandermonde determinant for the elements $\alpha_j , j \in L.$ Using a cofactor expansion for
$\det \eta_I$ yields an expression $\det ( h_I ; \alpha_L ) + \sum_{J < I} c_J \det ( h_J ; \alpha_L )$. The proposition therefore describes
a triangular transformation between the two sets of determinants $\{ \det \Sigma_I(p,q) \}_I$ and $\{ \det ( h_I ; \alpha_L ) \} _{I},$
with coefficients $c_J,$ for $J \leq I,$ that depend on $q(x)$ but not on $p(x)$. 

\begin{example}
For $E = \{ 0,1,2 \}$ and  $I = \{ 1,4 \}$, $q(x) = (x-\alpha_0)(x-\alpha_1)(x-\alpha_2).$ Let $p(x) = (x-\alpha_i)(x-\alpha_j)$.
\begin{align*}
 \det \begin{pmatrix} 1 &1 \\  \alpha_i  &\alpha_j \end{pmatrix} \det \Sigma_{\{1,4\}}(p,q) ~=~ &
          p_0^3 \det \begin{pmatrix} h_1(\alpha_i) &h_1(\alpha_j) &0 \\ h_3(\alpha_i) &h_3(\alpha_j) &q_0 \\ h_4(\alpha_i) &h_4(\alpha_j) &q_1 \end{pmatrix} \\
          ~=~ &- p_0^3 q_0 \det \begin{pmatrix} h_1(\alpha_i) &h_1(\alpha_j)  \\ h_4(\alpha_i) &h_4(\alpha_j) \end{pmatrix} 
                     + p_0^3 q_1 \det \begin{pmatrix} h_1(\alpha_i) &h_1(\alpha_j)  \\ h_3(\alpha_i) &h_3(\alpha_j) \end{pmatrix} \\
                ~=~ &- p_0^3 q_0 \det ( h_{\{1,4\}} ; \alpha_i, \alpha_j )  
                     + p_0^3 q_1 \det ( h_{\{1,3\}} ; \alpha_i, \alpha_j ).
\end{align*}
\end{example}

The map $\sigma_I$ has a doubly extended version
\begin{align*}
\sigma''_I ~:~ &V_0 \times V_0 \times  F[x]_{<m-v} \times F[x]_{<m-k} \longrightarrow V_0 \times V_0 \times V_1 \times V_2 \times V_2. \\
&~~(a_0,b_0,a,b) \mapsto (a_0 + \pi_0(ap), b_0 + \pi_0(bq), \sigma_I(a,b)).
\end{align*}

Determinants are the same for $\sigma_I$ and for the extended maps $\sigma'_I, \sigma''_I$, or the same up to sign if
bases are allowed to be reordered. The next example extends the matrix for $\sigma_I$ in Example~\ref{ex:si} to a matrix for $\sigma''_I.$

\begin{example} \label{ex:si2}
For $E = \{ 0,1,2 \}$ and  $I = \{ 1,4 \}$, $V_0 = \langle x \rangle, ~V_1 = \langle x^0, x^2, x^4 \rangle, ~V_2 =  \langle x^3 \rangle.$
\[
\Sigma''_I(p,q) = 
\bordermatrix{
&0 &1 &1 &2 &3 &3  &4 \cr \noalign{\smallskip}
&\cdot &1 &\cdot &\cdot &\cdot &\cdot &\cdot \cr
&\cdot &\cdot &1 &\cdot &\cdot &\cdot &\cdot \cr 
 \cline{2-8}
&p_2 &p_1 &\cdot &p_0 &\cdot &\cdot &\cdot \cr
&\cdot &p_2 &\cdot &p_1 &p_0 &\cdot &\cdot  \cr
&\cdot &\cdot &\cdot &p_2 &p_1 &\cdot &p_0 \cr \cline{2-8}
&q_3 &\cdot &q_2 &q_1 &\cdot &q_0 &\cdot  \cr
&\cdot &\cdot &q_3 &q_2 &\cdot &q_1 &q_0 \cr
}.
\]
\end{example}

\end{document}